\documentclass[preprint,numbers,sort&compress]{elsarticle}

\usepackage{amssymb}
\usepackage[figuresright]{rotating}
\usepackage{amsmath}
\usepackage[letterpaper,margin=1in]{geometry}
\usepackage{graphicx}
\usepackage{booktabs}
\usepackage{hyperref}
\usepackage{amsmath,amssymb,amsfonts}%
\usepackage{amsthm}%
\usepackage{mathrsfs}%
\usepackage[table]{ xcolor}
\usepackage{textcomp}%
\usepackage{pifont}
\usepackage{manyfoot}%
\usepackage{booktabs}%
\usepackage{algorithm}%
\usepackage{multirow}
\usepackage{subcaption}
\usepackage{algorithmicx}%
\usepackage{algpseudocode}%
\usepackage{listings}%
\newtheorem{lemma}{Lemma}
\newtheorem{remark}{Remark}
\newtheorem{definition}{Definition}

\newtheorem{theorem}{Theorem}

\usepackage[margin=1in]{geometry}
\usepackage{times}
\usepackage{hyperref}
\usepackage{graphicx}
\usepackage{float}
\usepackage{booktabs}
\usepackage{caption}
\journal{Journal of Computational Physics}

\begin{document}
	\begin{frontmatter}
		
		
		
		\title{Starter-Iterator Neural Operator: A Unified Architecture for High-Fidelity Forward and Inverse PDE Problems}
		
		\author[inst1]{Kuilin Qin}
		\author[inst1]{Lianfang Wang}
		\author[inst2]{Xu Sun}
		\author[inst2]{Jiwei Jia}
		\author[inst3]{Yu Wang}
		\author[inst4]{Yong Wang}
		\author[inst1]{Yuping Duan\corref{cor1}}
		
		\cortext[cor1]{Corresponding author}
		\ead{doveduan@gmail.com} 
		
		\affiliation[inst1]{organization={School of Mathematical Sciences, Beijing Normal University},
			postcode={100875},
			state={Beijing},
			country={P.R. China}}
		
		\affiliation[inst2]{organization={School of Mathematics, Jilin University},
			postcode={130015},
			state={Jilin},
			country={P.R. China}}
		
		\affiliation[inst3]{organization={Key Laboratory of Digital Technology in Medical Diagnostics of Zhejiang},
			postcode={130000},
			state={Zhejiang},
			country={P.R. China}}
		
		\affiliation[inst4]{organization={School of Physics, Nankai University},
			postcode={300071},
			state={Tianjin},
			country={P.R. China}}
		
		\begin{abstract}
			Operator learning is an emerging field at the intersection of machine
			learning and scientific computing. By learning mappings between function
			spaces, neural operators provide data-driven surrogate models for families
			of partial differential equations (PDEs). Once trained, these models can
			evaluate solution operators efficiently, making them suitable for
			many-query applications such as real-time prediction and parameter sweeps.
			However, maintaining high approximation accuracy and stable long-term
			predictions remains challenging for complex forward and inverse problems. To address these challenges, we propose the Starter-Iterator Neural
			Operator (SINO), which incorporates the initialization and residual-correction structures of classical iterative solvers into neural
			operator learning. The frequency-domain Starter captures dominant global
			spectral features and provides an informed initial approximation, while the
			latent-space Iterator applies successive residual-based corrections to
			refine local and multiscale solution structures. Experiments on
			representative time-dependent PDEs, including the Navier-Stokes and acoustic wave equations, together with applications to image super-resolution and weather forecasting, show that SINO achieves competitive accuracy and stable performance across the benchmarks considered in this work.
		\end{abstract}
		
		\begin{graphicalabstract}
			\centering
			\includegraphics[width=0.98\linewidth]{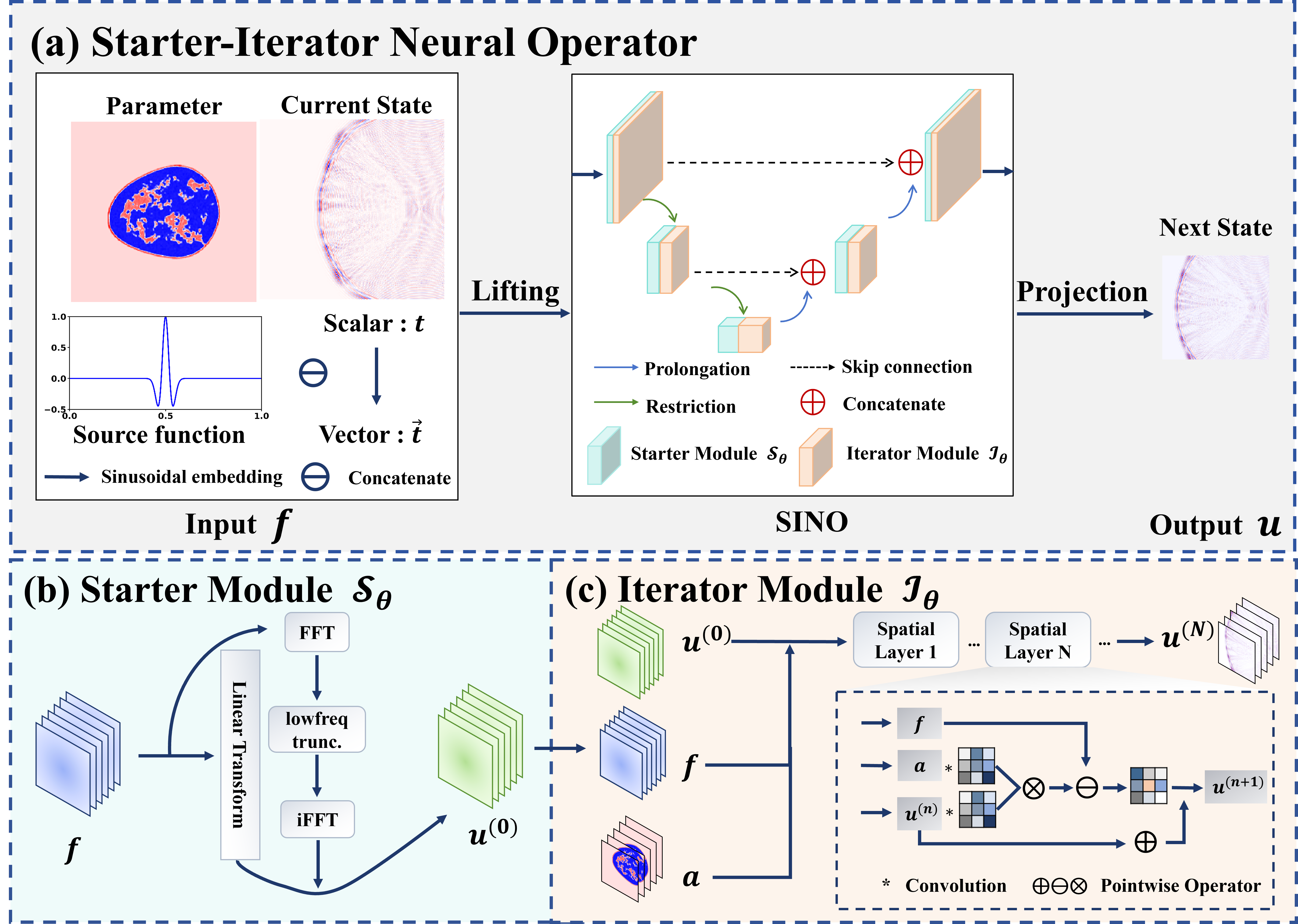}
		\end{graphicalabstract}
		
		\begin{highlights}
			\item We propose SINO, a neural operator architecture inspired by residual-correction schemes. For its abstract realization with a fixed linear operator equation, we establish a compact-uniform approximation result under boundedness and contraction assumptions.
			\item We characterize frequency-selective error damping under explicit common-eigenbasis and spectral-separation assumptions. This analysis provides a numerical interpretation of the complementary roles of spectral initialization and iterative residual refinement in representing multiscale solution features.
			\item We evaluate SINO on representative forward PDE problems, super-resolution inverse problems, and time-dependent forecasting tasks. Across the considered benchmarks, SINO achieves competitive or lower approximation errors than the compared baselines and exhibits stable performance over the evaluated prediction horizons.
		\end{highlights}
		
		
	\end{frontmatter}
	
		
		
		\section{Introduction}\label{sec1}
		
		Forward and inverse problems are fundamental to scientific and engineering
		computing \cite{li2025fully,tang2025optical,ding2025scitoolagent}. A forward
		problem predicts the state or response of a system from prescribed physical
		laws, parameters, initial conditions, and external forcing. Such problems
		arise in structural mechanics \cite{he2024multi}, computational acoustics
		\cite{tsakmakidis2025discovery}, and weather forecasting
		\cite{allen2025end,bauer2023deep}. An inverse problem instead uses
		observations to estimate unknown coefficients, sources, initial conditions,
		or latent fields \cite{aarts2025physics,tarantola2005inverse}. Representative
		applications include seismic tomography \cite{lyu2025efficient} and medical
		image reconstruction \cite{zhao2023energy}. In many scientific workflows,
		forward and inverse problems are closely coupled: estimated physical
		quantities are used in forward simulations, while discrepancies between
		predictions and observations are used to refine those estimates. 
		
		Forward problems are commonly formulated using differential equations \cite{carleo2019machine,brunton2024promising}, which describe local conservation laws and physical interactions such as force balance, mass conservation, and energy transport \cite{kontolati2024learning}. A general
		time-dependent problem can be written as
		\begin{equation}\label{pde}
			\left\{
			\begin{aligned}
				& \frac{\partial \boldsymbol{w}(\mathbf{x}, t)}{\partial t} + \mathcal{D}_a\boldsymbol{w}(\mathbf{x}, t) = \boldsymbol{s}(\mathbf{x}, t), & (\mathbf{x}, t) &\in \Omega \times (0, \infty), \\
				& \boldsymbol{w}(\mathbf{x}, 0) = \boldsymbol{w}_0(\mathbf{x}), & \mathbf{x} &\in \Omega, \\
				& \boldsymbol{w}(\mathbf{x}, t) = \boldsymbol{g}(\mathbf{x}, t), & \mathbf{x} &\in \partial \Omega,
			\end{aligned} 
			\right.
		\end{equation}
		where $\Omega\subset\mathbb R^{d}$ is the spatial domain,
		$\boldsymbol w$ is the state variable, $\boldsymbol w_0$ is the initial condition, $\boldsymbol g$ denotes the prescribed boundary data, and $\boldsymbol s$ is the source term. The spatial differential operator $\mathcal D_a$ is of order $m>0$ and may depend on a spatially varying coefficient field $a(\mathbf x)$. Depending on the governing equation, $\mathcal D_a$ may be linear or nonlinear in $\boldsymbol w$. For periodic problems, $\Omega$ is identified with the torus $\mathbb T^d$, and the boundary condition in Eq.~\eqref{pde} is replaced by periodicity.
		
		Forward problems may also be formulated using integral operators, which describe nonlocal interactions and cumulative effects \cite{zappala2024learning,georgiou2025fredholm}. A representative example is the Fredholm integral equation of the second kind
		\begin{equation}\label{integral}
			\boldsymbol{w}(\mathbf{x}) = \boldsymbol{s}(\mathbf{x}) + \lambda \int_{\Omega}\mathcal{K}(\mathbf{x},  \mathbf{y}) \boldsymbol{w}(\mathbf{y}) \, d\mathbf{y},
		\end{equation}
		where $\boldsymbol s$ is a prescribed forcing term, $\lambda$ is a scalar parameter, and $\mathcal K(\mathbf x,\mathbf y)$ is the kernel of the integral operator. The kernel determines how the state at $\mathbf y$ contributes to the response at $\mathbf x$ and plays a role analogous to a matrix in a finite-dimensional linear system. Integral formulations are particularly useful for representing nonlocal interactions,
		Green's-function solutions, and boundary-integral models.

		Inverse problems seek to recover unknown quantities from partial or noisy observations of the system response. For a differential model such as Eq.~\eqref{pde}, the unknown quantity may be the coefficient field $a(\mathbf x)$, the source $\boldsymbol s(\mathbf x,t)$, or the initial
		condition $\boldsymbol w_0$, while observations of
		$\boldsymbol w(\mathbf x,t)$ are available over a subset of the space-time domain ~\cite{azizzadenesheli2024neural}. In an integral-operator setting, the objective may be to reconstruct an unknown field or source density from integral measurements ~\cite{wu2022neural}. Such inverse problems are often ill posed because the forward measurement operator may be non-injective or sensitive to perturbations in the observations. Forward and inverse problems can nevertheless be expressed through mappings between function spaces, providing the common operator-learning perspective adopted in this work.

		Forward and inverse problems commonly use numerical discretization techniques~\cite{adler2025numerical} such as finite element methods, finite difference methods, finite volume methods , and spectral methods ~\cite{shen2011spectral}. These methods transform mathematical models describing continuous systems into discrete linear or nonlinear algebraic equations, laying the groundwork for subsequent solution using iterative numerical methods~\cite{zhang2024blending}. However, iterative methods are crucial for large-scale scientific computing but have notable limitations. Linear iterative methods~\cite{jiao2025one}, such as Jacobi ~\cite{zhang2024blending} and Gauss-Seidel ~\cite{ortega1966nonlinear}, rely heavily on matrix properties and are only effective for specific structures, struggling with slow convergence and high computational costs in ill-conditioned cases. Krylov subspace methods~\cite{saad1981krylov}, like conjugate gradient~\cite{hestenes1952methods} and generalized minimum residual method~\cite{saad1986gmres}, are limited by preconditioner design, with constructing an effective preconditioner being a major challenge. For nonlinear problems, Newton's method~\cite{sherman1978newton} offers only local convergence, is sensitive to initial guesses, and requires computing the Jacobian and solving a linear system at each step, leading to significant computational and memory demands.

		Neural operators have emerged as effective alternatives to traditional numerical solvers for solving classical forward and inverse problems~\cite{NeuralInverse, li2024physics, liu2025difffno, cao2025diff, wang2025noise}, by learning continuous mappings between infinite-dimensional Banach spaces \cite{kovachki2023neural}. 
		Recent advances~\cite{liu2026deep} in operator learning are deeply rooted in the finite basis representations of solution operators. From this perspective, a neural operator approximates the target mapping via a finite expansion, whose coefficients depend on the input function and whose basis functions span the output solution space. Consequently, the treatment of basis functions serves as the defining distinction among operator learning frameworks, categorizing the landscape into three primary classes. The first class comprises reduced-basis neural operators, which utilize basis functions pre-computed via classical model reduction techniques like Proper Orthogonal Decomposition (POD) and Principal Component Analysis (PCA). Representative frameworks include POD-NN~\cite{hesthaven2018non}, PCANet~\cite{bhattacharya2021model}, and DIPNet~\cite{o2022derivative}. By exploiting the low-rank structure of the solution manifold, these approaches reduce the learning task to mapping inputs to expansion coefficients. While offering robust numerical stability and high computational efficiency—especially for compactly representable solutions—their generalization is inherently bounded by the a priori fixed basis; performance typically degrades when test data deviates from the training manifold. To enhance flexibility and expressivity, a second class directly learns basis functions from data. Epitomized by DeepONet~\cite{LuDeepOnet} and its variants~\cite{he2024geom, xu2023transfer, lu2022comprehensive}, these methods employ branch-trunk architectures to co-optimize coefficient mappings and basis functions. By bypassing pre-computed POD/PCA modes, they achieve superior cross-domain generalization. However, lacking explicit structural constraints or physical priors, these learned bases can lead to diminished approximation accuracy and training instabilities when applied to complex, chaotic dynamical systems. 
		A third class employs prescribed functional bases, analytically defined based on underlying physics. Notable examples include the Fourier Neural Operator (FNO)~\cite{li2020fourier}, Spectral Neural Operator~\cite{fanaskov2023spectral}, Wavelet Neural Operator~\cite{tripura2023wavelet}, and Laplace Neural Operator~\cite{cao2024laplace}. These frameworks leverage integral transforms (Fourier, wavelet, or Laplace) to represent operators in the spectral domain. By embedding strong inductive biases, they achieve high efficiency for systems with well-defined spectral characteristics. In particular, FNO and its derivatives~\cite{li2025d, lehmann20243d, rahman2022uno} excel in PDEs dominated by low-frequency dynamics. Nonetheless, due to inherent spectral sampling limits and truncation effects, their performance often deteriorates under high-frequency regimes, compromising temporal stability in long-term autoregressive forecasting.
		
		Furthermore, architectures inspired by classical iterative algorithms offer a compelling alternative due to their inherent convergence guarantees. For example, MgNO \cite{he2024mgno, he2026self} incorporates multiscale iterative strategies and demonstrates performance superior to the aforementioned categories for several benchmark PDE problems.  However, MgNO typically requires a substantial number of parameters to achieve high accuracy due to its reliance on a zero-initialization scheme, and its emphasis on local features further limits its ability to capture the global dependency structures that are fundamental in many physical systems. The HINTs \cite{zhang2024blending} framework enhances convergence speed and accuracy by embedding a DeepONet-based solver within Jacobi iterations. Despite its effectiveness, this approach depends on explicit access to the underlying linear operator, which restricts its applicability to complex tasks such as super-resolution and data-driven weather prediction.

		In this paper, we present the Starter-Iterator Neural Operator (SINO), a
		framework designed to improve approximation accuracy and stability in
		long-horizon autoregressive prediction. SINO decomposes operator learning
		into two complementary stages. A frequency-based Starter captures the
		dominant spectral components of the solution and provides an informed initial approximation. A parameterized Iterator then applies successive
		residual-correction steps to refine the initial prediction and recover
		multiscale solution features. Experiments on the benchmarks considered in
		this work show that this decomposition improves approximation accuracy and
		maintains stable predictive performance over the evaluated autoregressive
		horizons. The main contributions are summarized as follows:   
		\begin{itemize}
			\item  We propose SINO, a neural operator architecture inspired by residual-correction schemes. For its abstract realization with a fixed linear operator equation, we establish a compact-uniform approximation result under boundedness and contraction assumptions.
			
			\item We characterize frequency-selective error damping under explicit common-eigenbasis and spectral-separation assumptions. This analysis provides a numerical interpretation of the complementary roles of spectral initialization and iterative residual refinement in representing multiscale solution features.
			
			\item We evaluate SINO on representative forward PDE problems, super-resolution inverse problems, and time-dependent forecasting tasks. Across the considered benchmarks, SINO achieves competitive or lower approximation errors than the compared baselines and exhibits stable performance over the evaluated prediction horizons.
		\end{itemize}
		
		The remainder of this paper is organized as follows.
		Section~\ref{sec2} introduces the problem formulation and mathematical
		motivation, describes the SINO architecture, and presents the corresponding
		theoretical results. Section~\ref{sec3} reports the numerical experiments,
		including an empirical analysis of the spectral behavior of SINO and
		evaluations on PDE, microscopy, and meteorological datasets.
		Section~\ref{sec4} summarizes the main findings, discusses the limitations
		of the present study, and outlines directions for future work. The principal
		notations used throughout the paper are summarized in
		Table~\ref{tab:notations}.
		
		\begin{table}[t]
			\centering
			\caption{Summary of notations used in this paper.}
			\label{tab:notations}
			\begin{tabular}{ll}
				\hline
				\textbf{Symbol} & \textbf{Description} \\ \hline
				$\Omega$ & Physical spatial domain. \\
				$\mathbb{T}^d$ & periodic torus, identified with $[0, 2\pi]^d$. \\
				$\boldsymbol{f}, \boldsymbol{u}$ & Input source function and target solution function in the physical space. \\
				$\mathcal{X}, \mathcal{Y}$ & Input and output function spaces, i.e., Sobolev space { {$H^{r-m}(\mathbb{T}^d; \mathbb{R}^{d_0})$ and $H^r(\mathbb{T}^d; \mathbb{R}^{d_0})$.}} \\
				$\mathcal{D}_a$ & Spatial differential operator of order $m$ in Eq.~\eqref{pde}, which can be linear or nonlinear.  \\
				$\mathcal L_{a}$ & Linearized the nonlinear spatial operator $\mathcal{D}_a$ with Picard-type representation. \\
				$\mathcal G$ & Exact solution operator with $\boldsymbol u=\mathcal G(\boldsymbol f)$. \\
				$\mathcal G_\theta$ & Parameterized neural approximation of $\mathcal G$. \\
				$\mathcal A$ & Linear or approximately linearized  operator that admits general relation $\mathcal A\boldsymbol u=\boldsymbol f$. \\
				$\mathcal A_{\theta^{(n)}}$ & Parameterized approximation of $\mathcal A$ at the $n$-th refinement step. \\
				$\mathcal B$ & Preconditioning or correction operator used in the fixed-point iteration. \\
				$\mathcal B_{\theta^{(n)}}$ & Parameterized correction operator at the $n$-th refinement step. \\
				
				$\mathrm{Id}_{\mathcal X},\mathrm{Id}_{\mathcal Y}$ & Identity operator on $\mathcal X$ and $\mathcal Y$ space, respectively. \\
				
				$\mathcal S$ & Starter operator that provides an initial approximation. \\
				
				$\mathcal I$ & Single step iterative operator that performs iterative refinement. \\
				
				$\mathcal I^{N}$ &  Iterative operator with $N$-steps iteration even infinite steps as $N=\infty$. \\
				
				$\mathcal S_\theta$ & Learnable starter operator, typically implemented by a spectral neural operator. \\
				
				$\mathcal I_{\theta^{(n)}}$ & Learnable single-step iterative operator at $n$-th iteration step \\
				$\mathcal I_\theta^{N}$ & Learnable multi-steps iterative operator with $N$ steps. \\
				$\mathcal P_K$ & Fourier projection operator that retains modes up to the cutoff $K$. \\
				
				$\mathcal G_K$ & Fourier-truncated approximation of $\mathcal G$ using modes up to $K$. \\
				
				$\mathcal E_{{\theta}^{(n)}}$ & Error-propagation operator denoted as
				$\mathcal  E_{{\theta}^{(n)}}=\mathrm{Id}_{\mathcal Y}-\mathcal B_{\theta^{(n)}}\mathcal A_{\theta^{(n)}}$. \\
				
				$\mathcal R_s$ & Restriction operator from scale $s$ to a coarser scale. \\
				
				$\mathcal T_s$ & Prolongation or transposed-convolution operator from scale $s+1$ to a finer scale. \\
				$\mathcal{L}_{\theta}, \mathcal{Q}_{\theta}$ & Lifting and projection block of neural operator.  \\
				
				$R_\theta$ & Learnable complex-valued spectral weights. \\
				
				$W_\theta$ & Pointwise linear transformation in the neural network. \\
				
				$\sigma(\cdot)$ & Nonlinear activation function, e.g., GELU or ReLU. \\
				\hline
			\end{tabular}
		\end{table}
		
		\section{Our Methodology}\label{sec2}
		
		\subsection{Linearization of Forward and Inverse Problems}\label{sec:linearization}
		Consider the forward problem of solving for $\boldsymbol w(\cdot,t)$ governed by Eq.~\eqref{pde}. We assume that $\boldsymbol w(\cdot,t)\in \mathcal Y:=H^r(\mathbb T^d;\mathbb R^{d_{0}})$, $r>m>0$ and define $\mathcal X:=H^{r-m}(\mathbb T^d;\mathbb R^{d_{0}})$. Since $\mathcal Y$ is continuously embedded in $\mathcal X$, throughout this section we identify each element of $\mathcal Y$ with its image in $\mathcal X$. Given a constant time step $\Delta t>0$, the backward Euler semi-discretization gives
		\[
		\frac{
			\boldsymbol w(t+\Delta t)-\boldsymbol w(t)
		}{\Delta t}
		+
		\mathcal D_a\bigl(\boldsymbol w(t+\Delta t)\bigr)
		=
		\boldsymbol s(t+\Delta t),
		\qquad t>0,
		\]
		where the equation is understood in $\mathcal X$. By rearranging, it gives
		\begin{equation}
			\label{eq:implicit_step}
			\boldsymbol w(t+\Delta t)
			+
			\Delta t\,
			\mathcal D_a\bigl(\boldsymbol w(t+\Delta t)\bigr)
			=
			\boldsymbol w(t)
			+
			\Delta t\,\boldsymbol s(t+\Delta t).
		\end{equation}
		If $\mathcal D_a$ is linear, we assume $\mathcal D_a\in\mathcal L(\mathcal Y,\mathcal X)$. The corresponding implicit operator $\mathcal{A}$ is defined by
		\[
		\mathcal A\boldsymbol u
		:=
		\boldsymbol u+\Delta t\,\mathcal D_a\boldsymbol u,
		\qquad
		\boldsymbol u\in\mathcal Y.
		\]
		Under the continuous-embedding convention above, $\mathcal A\in\mathcal L(\mathcal Y,\mathcal X)$. Existence and uniqueness of the implicit solution require the corresponding operator equation to be well posed.
		
		If $\mathcal D_a$ is nonlinear, Eq.~\eqref{eq:implicit_step} requires the
		solution of a nonlinear system at each time step. To obtain a linear
		subproblem, we employ a Picard-type lagged-coefficient linearization
		\cite{deblois1997linearizing}. We assume that
		$\mathcal L_a\bigl(\boldsymbol w(t)\bigr)
		\in \mathcal L(\mathcal Y,\mathcal X)$ is a bounded linear operator satisfying
		\[
		\mathcal D_a\bigl(\boldsymbol w(t+\Delta t)\bigr)
		\approx
		\mathcal L_a\bigl(\boldsymbol w(t)\bigr)
		\boldsymbol w(t+\Delta t).
		\]
		Substituting this approximation into the backward Euler scheme gives
		\[
		\boldsymbol w(t+\Delta t)
		+
		\Delta t\,
		\mathcal L_a\bigl(\boldsymbol w(t)\bigr)
		\boldsymbol w(t+\Delta t)
		=
		\boldsymbol w(t)
		+
		\Delta t\,\boldsymbol s(t+\Delta t).
		\]
		Define $\boldsymbol u:=
		\boldsymbol w(t+\Delta t)\in\mathcal Y$ and $\boldsymbol f
		:=\boldsymbol w(t)+\Delta t\,\boldsymbol s(t+\Delta t)
		\in\mathcal X$. Both the linear problem and an individual frozen linearized subproblem can then be written as $\mathcal A_t:\mathcal Y\to\mathcal X$
		\begin{equation}
			\label{eq:linear}
			\mathcal A_t\boldsymbol u=\boldsymbol f,
		\end{equation}
		where, for any $\boldsymbol u\in\mathcal Y$,
		\begin{equation}
			\mathcal A_t\boldsymbol u
			=
			\begin{cases}
				\boldsymbol u+\Delta t\,\mathcal D_a\boldsymbol u,
				&
				\text{if $\mathcal D_a$ is linear},
				\\[1mm]
				\boldsymbol u
				+
				\Delta t\,
				\mathcal L_a\bigl(\boldsymbol w(t)\bigr)\boldsymbol u,
				&
				\text{if $\mathcal D_a$ is nonlinear and Picard-linearized}.
			\end{cases}
		\end{equation}
		At each time level, Eq.~\eqref{eq:linear} therefore represents either a linear implicit step or a frozen-coefficient approximation of a nonlinear implicit step. When analyzing one such fixed operator equation, we write $\mathcal A$ in place of $\mathcal A_t$. The theory developed below applies to this fixed equation under the stated boundedness and contraction
		assumptions; it does not provide a convergence result for the complete state-dependent nonlinear process.
		
		A linear inverse problem associated with the integral formulation in Eq.~\eqref{integral} can be expressed in the same algebraic form $\mathcal A\boldsymbol u=\boldsymbol f$, where $\mathcal A:\mathcal Y\to\mathcal X$ is the forward measurement operator, $\boldsymbol u$ is the unknown field, and $\boldsymbol f$ is the
		observed data. Thus, Eq.~\eqref{eq:linear} provides a common algebraic representation for forward simulation and linear inverse reconstruction. This representation alone does not imply existence, uniqueness, or stability of the solution; these properties require additional assumptions on $\mathcal A$.

		\subsection{Starter-Iterator Framework for the Linear Operator System}
		
		We seek a functional representation suitable for parameterizing the linear or frozen-linearized operator equation \eqref{eq:linear}. Motivated by classical stationary iterative methods, we approximate its solution through
		successive residual corrections. Throughout the analysis, $\mathcal A\in\mathcal L(\mathcal Y,\mathcal X)$ denotes a fixed bounded linear operator associated with Eq.~\eqref{eq:linear}. For a nonlinear problem, $\mathcal A$ may represent the operator obtained from one frozen linearized subproblem described in
		Section~\ref{sec:linearization}.
		
		Given an initial approximation $\boldsymbol u^{(0)}\in\mathcal Y$, the iterates $\{\boldsymbol u^{(n)}\}_{n=0}^{\infty}$ are defined by
		\begin{equation}
			\label{eq:iter_format}
			\boldsymbol u^{(n+1)}
			=
			\boldsymbol u^{(n)}
			+
			\mathcal B
			\left(
			\boldsymbol f-\mathcal A\boldsymbol u^{(n)}
			\right),
			\qquad n=0,1,2,\ldots,
		\end{equation}
		where $\mathcal B\in\mathcal L(\mathcal X,\mathcal Y)$ 
		is a bounded correction or preconditioning operator. If $\mathcal A$ is continuously invertible, $\mathcal B$ may be interpreted as an approximation to $\mathcal A^{-1}$. The quantity
		$\boldsymbol r^{(n)} := \boldsymbol f-\mathcal A\boldsymbol u^{(n)}
		\in\mathcal X$ 
		is the residual at the $n$-th iteration. Let $\boldsymbol u^*\in\mathcal Y$ satisfy
		$\mathcal A\boldsymbol u^*=\boldsymbol f$, and define $\boldsymbol e^{(n)}:= \boldsymbol u^*-\boldsymbol u^{(n)}$. 
		It follows from Eq.~\eqref{eq:iter_format} that 
		\[
		\boldsymbol e^{(n+1)}
		=
		\mathcal E\boldsymbol e^{(n)},
		\qquad
		\mathcal E
		:=
		\operatorname{Id}_{\mathcal Y}
		-
		\mathcal B\mathcal A
		\in\mathcal L(\mathcal Y),
		\]
		where $\mathcal E$ is the error-propagation operator.
		For finite-dimensional stationary iterations, convergence for every initial
		state is characterized by the spectral-radius condition
		$\rho(\mathcal E)<1$. In the present Hilbert-space analysis, we impose the
		stronger sufficient condition $\|\mathcal E\|_{\mathcal L(\mathcal Y)}<1$. This condition gives the geometric estimate $\|\boldsymbol e^{(n)}\|_{\mathcal Y}
		\leq
		\|\mathcal E\|_{\mathcal L(\mathcal Y)}^n
		\|\boldsymbol e^{(0)}\|_{\mathcal Y}$ 
		and therefore guarantees convergence to the exact solution. The contraction condition also implies that
		$\mathcal A$ is bounded below and hence injective. Indeed, for every $\boldsymbol u\in\mathcal Y$,
		\[
		(1-q)\|\boldsymbol u\|_{\mathcal Y}
		\le
		\|\mathcal B\mathcal A\boldsymbol u\|_{\mathcal Y}
		\le
		\|\mathcal B\|_{\mathcal L(\mathcal X,\mathcal Y)}
		\|\mathcal A\boldsymbol u\|_{\mathcal X},
		\qquad
		q:=\|\operatorname{Id}_{\mathcal Y}
		-\mathcal B\mathcal A\|_{\mathcal L(\mathcal Y)}<1.
		\]
		Consequently, $\mathcal A$ has a continuous inverse on its range, and the solution operator
		\[
		\mathcal G: \operatorname{Ran}(\mathcal A)\to\mathcal Y,
		\qquad
		\mathcal A\mathcal G(f)=f,
		\]
		is well defined and continuous.
		
		We now separate the initialization and refinement stages of the iteration. Let 
		\[
		\mathcal S:\mathcal X\to\mathcal Y,
		\qquad
		\boldsymbol u^{(0)}
		=
		\mathcal S(\boldsymbol f),
		\]
		denote the \emph{Starter}, which provides the initial approximation. The single-step \emph{Iterator} is defined by
		\begin{equation}
			\label{eq:one_iter}
			\mathcal I: \mathcal Y\times\mathcal X \to \mathcal Y,
			\qquad
			\mathcal I(\boldsymbol u,\boldsymbol f):=
			\boldsymbol u
			+
			\mathcal B\left(\boldsymbol f-\mathcal A\boldsymbol u\right).
		\end{equation}
		For fixed $\boldsymbol f\in\operatorname{Ran}(\mathcal A)$, define
		\[
		\mathcal I^\infty(\boldsymbol u;\boldsymbol f)
		:=
		\lim_{n\to\infty}
		\mathcal I^n(\boldsymbol u;\boldsymbol f),
		\]
		where $\mathcal I^n$ denotes $n$ successive applications of
		Eq.~\eqref{eq:one_iter} with $\boldsymbol f$ held fixed. Under
		contraction condition, this limit is independent of the initial state and satisfies
		\[
		\mathcal I^\infty(\boldsymbol u;\boldsymbol f)
		=
		\mathcal G(\boldsymbol f).
		\]
		Therefore, the solution operator admits the Starter--Iterator
		representation
		\[
		\mathcal G(\boldsymbol f)
		=
		\mathcal I^\infty
		\left(
		\mathcal S(\boldsymbol f);
		\boldsymbol f
		\right),
		\qquad
		\boldsymbol f\in\operatorname{Ran}(\mathcal A).
		\]
		The Starter controls the initial error, whereas the Iterator reduces that error through successive residual corrections. Accordingly, the finite-step accuracy depends jointly on the quality of the initialization and the contraction behavior of the refinement process.

		\subsection{Starter-Iterator Neural Operator}
		\label{math_form}
		
		We introduce the Starter-Iterator Neural Operator (SINO), which connects classical iterative solvers with neural operator learning. The SINO architecture, denoted by $\mathcal G_\theta$, consists of a Starter operator $\mathcal S_\theta$ and an $N$-step learned Iterator $\mathcal I_\theta^N$. These two components correspond to the initialization and iterative-refinement stages of numerical solvers, respectively. Here,
		$\theta$ collectively denotes all trainable parameters of SINO, while $\theta^{(n)}$ denotes the parameters used at the $n$-th iteration step.
		
		\textbf{Starter operator.} The Starter aims to construct an informative initial approximation of the target before the subsequent iterative refinement. 
		As global basis functions provide an efficient representation of spatially correlated solution structures,  transforming the input into the frequency domain enables the model to directly learn interactions among global Fourier modes, thereby capturing the dominant global structures of the underlying solution operator. \\
		Accordingly, given the input data $\boldsymbol f\in\mathcal X$, a representative spectral layer of the Starter is defined as
		\begin{equation*}
			\mathcal{S}_{\theta} : \mathcal{X}
			\rightarrow \mathcal{Y}, \qquad \mathcal{S}_{\theta} (\boldsymbol f) :=
			\sigma\!\left(
			W_\theta\boldsymbol f
			+\mathcal F^{-1}
			\left[R_\theta(k)\mathcal F(\boldsymbol f)(k)\right]
			\right).
		\end{equation*}
		where $\mathcal F$ and $\mathcal F^{-1}$ are the Fourier transform
		and its inverse, $R_\theta(k)$ contains the learnable spectral
		weights, $W_\theta$ is a pointwise linear transformation. 
		For the approximation analysis, the Starter is chosen from the
		FNO class in ~\cite{kovachki2021universal}, with a smooth, globally Lipschitz,
		non-polynomial activation and sufficient network capacity.
		
		\textbf{Iterator operator.}
		The starter produces the initial approximation $\boldsymbol u^{(0)}=\mathcal S_\theta(\boldsymbol f)$. This approximation is subsequently refined through the residual-correction scheme
		\begin{equation}
			\label{eq:learned_iteration}
			\boldsymbol u^{(n+1)}
			=
			\boldsymbol u^{(n)}
			+
			\mathcal B_{\theta^{(n)}}\left(
			\boldsymbol f
			-
			\mathcal A_{\theta^{(n)}}\boldsymbol u^{(n)}
			\right),
			\qquad
			n=0,\ldots,N-1,
		\end{equation}
		where $\mathcal A_{\theta^{(n)}}:\mathcal Y\to\mathcal X$ and $\mathcal B_{\theta^{(n)}}:\mathcal X\to\mathcal Y$ are parameterized approximations of the forward and preconditioning operators, respectively. In the theoretical analysis, these operators are restricted to bounded linear operators such that $\mathcal A_{\theta^{(n)}}
		\in\mathcal L(\mathcal Y,\mathcal X)$ and $\mathcal B_{\theta^{(n)}}
		\in\mathcal L(\mathcal X,\mathcal Y)$. 
		Under this restriction, the $n$-th learned iteration step is the affine map $\mathcal I_{\theta^{(n)}}:
		\mathcal Y\times\mathcal X\to\mathcal Y$ defined by
		\begin{equation}
			\label{eq:learned_iterator_step}
			\mathcal I_{\theta^{(n)}}(\boldsymbol u,\boldsymbol f)
			:=
			\mathcal E_{\theta^{(n)}}\boldsymbol u
			+
			\mathcal B_{\theta^{(n)}}\boldsymbol f,
		\end{equation}
		where $\mathcal E_{\theta^{(n)}}
		:=\operatorname{Id}_{\mathcal Y}-
		\mathcal B_{\theta^{(n)}}\mathcal A_{\theta^{(n)}}
		\in\mathcal L(\mathcal Y)$ is the learned error-propagation operator at iteration $n$. For integers $b\geq a$, define
		\[
		\mathcal E_{\theta^{[b:a]}}
		:=
		\mathcal E_{\theta^{(b)}}
		\circ
		\mathcal E_{\theta^{(b-1)}}
		\circ\cdots\circ
		\mathcal E_{\theta^{(a)}}.
		\]
		For $b<a$, we adopt the empty-composition convention $\mathcal E_{\theta^{[b:a]}}:=\operatorname{Id}_{\mathcal Y}$. The $N$-step learned Iterator is defined recursively by
		$\mathcal I_\theta^0(\boldsymbol u;\boldsymbol f)
		:=\boldsymbol u$ and 
		\begin{equation}
			\label{eq:multi_step_learnable_iterator}
			\mathcal I_\theta^{n+1}(\boldsymbol u;\boldsymbol f)
			:=
			\mathcal I_{\theta^{(n)}}\left(
			\mathcal I_\theta^n(\boldsymbol u;\boldsymbol f),
			\boldsymbol f
			\right),
			\qquad
			n=0,\ldots,N-1.
		\end{equation}
		Thus, $\mathcal I_\theta^N$ denotes the ordered application of $N$ learned iteration steps with iteration-dependent parameters. Parameter sharing across the iteration steps is not assumed.  In the bounded-linear setting, unrolling
		Eq.~\eqref{eq:learned_iteration} gives 
		\begin{equation}
			\label{eq:iterator}
			\boldsymbol u^{(N)}
			=
			\mathcal I_\theta^N
			\left(
			\boldsymbol u^{(0)};\boldsymbol f
			\right)=
			\mathcal E_{\theta^{[N-1:0]}}
			\boldsymbol u^{(0)}
			+
			\sum_{i=0}^{N-1}
			\mathcal E_{\theta^{[N-1:i+1]}}
			\mathcal B_{\theta^{(i)}}\boldsymbol f.
		\end{equation}
		Finally, the complete SINO is defined by
		\begin{equation}
			\label{sino_form}
			\mathcal G_\theta
			=
			\mathcal I_\theta^N
			\circ
			\left(
			\mathcal S_\theta,
			\operatorname{Id}_{\mathcal X}
			\right).
		\end{equation}
		Equivalently, it gives $\mathcal G_\theta(\boldsymbol f)
		=\mathcal I_\theta^N\left(\mathcal S_\theta(\boldsymbol f);
		\boldsymbol f \right)$. 
		As can be observed, the Starter provides a spectral initialization, while the Iterator performs successive residual-based corrections. The approximation property of this decomposition is established in the following theorem.

		\begin{theorem}[Compact-uniform approximation of the abstract Starter–Iterator decomposition]
			\label{thm:fsno_approx}
			Let
			\[
			\mathcal Y=H^r(\mathbb T^d;\mathbb R^{d_{0}}),
			\qquad
			\mathcal X=H^{r-m}(\mathbb T^d;\mathbb R^{d_{0}}),
			\qquad r>m>0,
			\]
			and let $\mathcal C\subset \operatorname{Ran}(\mathcal A)$
			be compact with respect to the \(\mathcal X\)-norm. Assume that $\mathcal A\in\mathcal L(\mathcal Y,\mathcal X)$
			and that there exists
			$\mathcal B\in\mathcal L(\mathcal X,\mathcal Y)$
			such that
			\[
			q:=\|\operatorname{Id}_{\mathcal Y}-\mathcal B\mathcal A\|_
			{\mathcal L(\mathcal Y)} \in (0, 1).
			\]
			Then $\mathcal A$ is injective and bounded below on $\mathcal Y$, $\operatorname{Ran}(\mathcal A)$ is closed in $\mathcal X$.
			Therefore the solution map
			\[
			\mathcal G: \operatorname{Ran}(\mathcal A)\to\mathcal Y,
			\qquad
			\mathcal A\mathcal G(f)=f,
			\]
			is a well defined and bounded linear operator. For every $\epsilon>0$, there exist an iteration number $N\in\mathbb N_+$, a continuous spectral Starter
			$\mathcal S_\theta:\mathcal X\to\mathcal Y$, and finite-rank bounded
			linear operators
			\[
			\mathcal A_{\theta^{(n)}}\in\mathcal L(\mathcal Y,\mathcal X),
			\qquad
			\mathcal B_{\theta^{(n)}}\in\mathcal L(\mathcal X,\mathcal Y),
			\qquad
			n=0,\ldots,N-1,
			\]
			such that the corresponding SINO $\mathcal G_\theta(\boldsymbol f)
			:= \mathcal I_\theta^N
			\left(\mathcal S_\theta(\boldsymbol f);
			\boldsymbol f\right)$ satisfies
			\[
			\sup_{\boldsymbol f\in\mathcal C}
			\left\|
			\mathcal G(\boldsymbol f)
			-
			\mathcal G_\theta(\boldsymbol f)
			\right\|_{\mathcal Y}
			<\epsilon.
			\]
		\end{theorem}
		
		{\begin{remark}[Scope of the theoretical result]
				\label{rem:theoretical_scope}
				Theorem~\ref{thm:fsno_approx} establishes compact-uniform
				approximation for a fixed bounded linear operator equation.
				Its contraction assumption implies that $\mathcal A$ is injective
				and bounded below. The result therefore does not cover the complete
				state-dependent nonlinear Picard process or underdetermined inverse
				problems with a nontrivial null space. The corresponding experiments
				evaluate the architecture empirically.
			\end{remark}
			
			\begin{remark}[Theoretical and practical Iterators]\label{rem:fcn_cnn_relation}
				Bounded linear operators between separable Hilbert spaces admit
				finite-rank approximations that converge uniformly on compact sets.
				After finite-dimensional coefficient encoding and decoding,
				these approximations can be represented exactly by fully connected
				linear layers without biases or nonlinear activations.
				Convolutional layers instead impose local connectivity and weight
				sharing, providing an efficient parameterization of local spatial
				interactions associated with differential operators.
				The practical Iterator combines convolutional and multiscale
				blocks with nonlinear operations; its performance is evaluated
				empirically and is not directly covered by the finite-rank linear
				analysis.
		\end{remark}}

		\begin{figure}[t]
			\centering
			\includegraphics[width=0.8\linewidth]{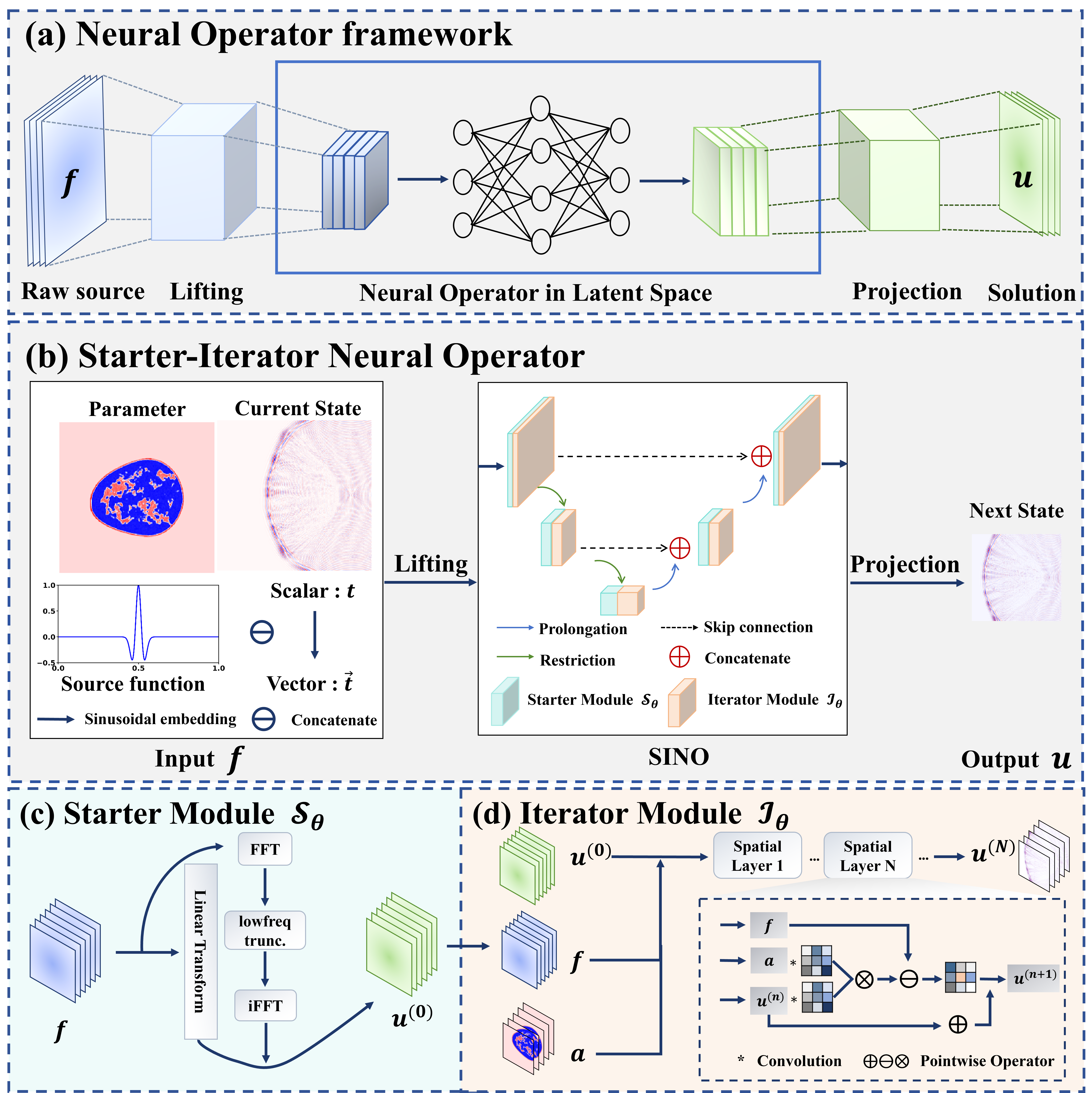}\\
			\vspace{0.3cm}
			\caption{\textbf{Illustration of the proposed starter-iterator neural operator.}
				(\textbf{a}) General framework of operator learning, where the input in space $\mathcal{X}$ is lifted into latent source space, processed by the neural operator, amd projected to the solution space $\mathcal{Y}$, and projected to true solution space as the output. 
				(\textbf{b}) Framework of our proposed Starter-Iterator Neural Operator (SINO), which incorporates both Starter and Iterator modules, together with lifting, projection, restriction, and prolongation operators. 
				(\textbf{c}) The starter operator $\mathcal{S}_{\theta}$ applies Fourier transform, low-frequency truncation, and inverse transform to capture global patterns. 
				(\textbf{d}) The iterator operator $\mathcal{I}_{\theta}$ enhances local structures via spatial filtering layers and convolutional operators.}
			\label{fig1:neuraloperator}
		\end{figure}
		
		\subsection{Network Architecture}
		Motivated by the Starter-Iterator decomposition and the residual-correction structure introduced in Section~\ref{math_form}, we implement SINO using the hierarchical multiscale architecture illustrated in Figures~\ref{fig1:neuraloperator}(b-d). The
		architecture combines complementary spectral and spatial operations across multiple discretization scales. The Starter uses learnable Fourier filtering to capture dominant global structures and long-range dependencies, thereby providing the initial approximation $\boldsymbol u^{(0)}$. The Iterator then applies localized convolutional blocks to learn successive residual-based corrections and recover finer spatial features. This multiscale design is
		motivated by the correction structure of classical iterative solvers.
		
		\textbf{Time-variable embedding.}
		For time-dependent problems, temporal information is incorporated through a high-dimensional embedding of the time variable $t\geq0$. Following the positional-encoding strategy used in attention and diffusion models
		\cite{vaswani2017attention,ho2020denoising}, we employ sinusoidal functions with geometrically spaced frequencies:
		\[
		\boldsymbol\gamma(t)
		=
		\left[
		\sin(\omega_0t),\cos(\omega_0t),
		\ldots,
		\sin(\omega_{M-1}t),\cos(\omega_{M-1}t)
		\right],
		\qquad
		\omega_j=\omega_0 2^j,
		\quad j=0,\ldots,M-1.
		\]
		The resulting vector $\boldsymbol\gamma(t)\in\mathbb R^{2M}$ is passed through a shallow MLP consisting of linear layers and ReLU activations to obtain
		$\boldsymbol e_t\in\mathbb R^{C_e}$. This vector is broadcast over the spatial grid to form $\widetilde{\boldsymbol e}_t
		\in \mathbb R^{B\times C_e\times H\times W}$, where $B$ denotes the batch size, $C_e$ is the embedding dimension, and $H\times W$ denotes the spatial resolution. Given a state tensor $\boldsymbol u_t \in \mathbb R^{B\times C_u\times H\times W}$, the temporally augmented input is obtained through channel-wise concatenation:
		\[
		\boldsymbol f_t^{\mathrm{aug}}
		:=
		\operatorname{Concat}
		\left(
		\boldsymbol u_t,
		\widetilde{\boldsymbol e}_t
		\right)
		\in
		\mathbb R^{B\times(C_u+C_e)\times H\times W}.
		\]
		Additional problem-dependent inputs, such as physical parameters or forcing terms, can be concatenated in the same manner. For stationary forward and inverse problems, the time embedding is omitted and the original problem
		input is used directly.
		
		\textbf{Latent-space embedding.} As illustrated in Figure~\ref{fig1:neuraloperator}(a), the practical
		architecture follows the lifting-operator-projection structure commonly used in neural operators \cite{kovachki2023neural}. The input $\boldsymbol f\in\mathcal X$ is first mapped by a pointwise lifting operator $\mathcal L_\theta$ to a higher-dimensional feature representation. The Starter and Iterator operate on the lifted features, after which a projection operator $\mathcal Q_\theta$ maps the refined representation to the physical solution space $\mathcal Y$. The practical network can therefore be written schematically as
		\[
		\mathcal G_\theta(\boldsymbol f)
		=
		\mathcal Q_\theta
		\left[
		\mathcal I_\theta^N
		\left(
		\mathcal S_\theta
		\left(
		\mathcal L_\theta(\boldsymbol f)
		\right);
		\mathcal L_\theta(\boldsymbol f)
		\right)
		\right].
		\]
		Both $\mathcal L_\theta$ and $\mathcal Q_\theta$ are implemented as shallow pointwise MLPs with ReLU activations. These transformations are applied
		independently at each spatial location and increase the channel capacity of the internal feature representation. The lifting and projection operations are implementation-level components and are absorbed into the abstract
		Starter and Iterator notation used in Section~\ref{math_form}.


		\textbf{Hierarchical Multiscale Framework.}
		Inspired by multigrid correction schemes \cite{xu1992iterative}, we adopt a hierarchical multiscale framework to capture complementary global and fine-scale solution structures. We use $S$ resolution levels indexed by
		$s=0,\ldots,S-1$, where $s=0$ denotes the finest level and $s=S-1$ denotes the coarsest level. Let $\boldsymbol f_s$ denote the right-hand side supplied to level $s$. Here, $\boldsymbol f_0$ is the original input, while
		$\boldsymbol f_s$ for $s>0$ is obtained by restricting the residual from the preceding finer level. At each level, the Starter-Iterator sequence produces an approximate
		level-wise solution or correction:
		\begin{equation}
			\label{eq:dual-domain}
			\boldsymbol u_s^{(0)}
			=
			\mathcal S_\theta(\boldsymbol f_s),
			\qquad
			\boldsymbol u_s^{(N)}
			=
			\mathcal I_\theta^N
			\big(
			\boldsymbol u_s^{(0)};
			\boldsymbol f_s
			\big).
		\end{equation}
		For simplicity, the possible scale dependence of the network parameters is suppressed in the notation. The Starter and Iterator serve as learned initialization and refinement blocks, respectively, at each resolution level.  
		During the downward pass, the level-$s$ residual is defined by $\boldsymbol r_s = \boldsymbol f_s - \mathcal A_{\theta,s}\boldsymbol u_s^{(N)}$, where $\mathcal A_{\theta,s}$ is a learned latent operator
		at scale $s$. For $s<S-1$, the residual is transferred to the next coarser level through the restriction operator $\mathcal R_s$, implemented using a strided
		convolution:
		$\boldsymbol f_{s+1}=\mathcal R_s(\boldsymbol r_s)$. 
		The residual equation is approximated at the coarsest level $s=S-1$.
		
		During the upward pass, the correction obtained at level $s+1$ is mapped to the finer level $s$ by the prolongation operator $\mathcal T_s$, implemented using a transposed convolution, and is concatenated to the finer-level approximation:
		\[
		\boldsymbol u_s^{(N)}
		\leftarrow
		\text{Concat} \big(\boldsymbol{u}^{(N)}_{s},  \mathcal T_s (\boldsymbol u_{s+1}^{(N)})\big) ,
		\qquad
		s=S-2,\ldots,0.
		\]
		This V-cycle structure captures long-range interactions on coarse spatial grids at a reduced computational cost, while the finer levels retain more spatial and spectral detail. In particular, the number of retained Fourier modes can be increased toward the finer levels. Cross-scale skip connections facilitate information flow between resolutions and are empirically observed to improve iterative refinement, as summarized in Algorithm~\ref{alg:multiscale}.

		\begin{algorithm}[htbp]
			\caption{Data flow in Starter-Iterator Neural Operator}
			\label{alg:multiscale}
			\begin{algorithmic}[1]
				\Require Finest-level input $\boldsymbol f_0$ and number of scales $S$
				\For{$s=0,\ldots,S-1$} \Comment{Downward pass: restriction and coarse solve}
				\State Apply the Starter and Iterator:
				\begin{equation*}
					\boldsymbol u_s^{(0)}
					\gets \mathcal S_\theta(\boldsymbol f_s) \quad \mbox{and}\quad
					\boldsymbol u_s^{(N)}
					\gets\mathcal I_\theta^N
					\big(\boldsymbol u_s^{(0)};\boldsymbol f_s\big),
				\end{equation*}
				\If{$s < S-1$}
				\State Compute residual: $\boldsymbol r_s
				\gets \boldsymbol f_s-\mathcal A_{\theta,s}\boldsymbol u_s^{(N)}$,
				\State Restrict residual to coarse grid: $\boldsymbol f_{s+1}\gets \mathcal R_s(\boldsymbol r_s)$,
				\EndIf
				\EndFor
				\For{$s=S-2,S-3,\ldots,0$} \Comment{Upward pass: prolongation and correction}
				\State Prolongation and Correction: $\boldsymbol{u}^{(N)}_{s} \gets \text{Concate}\big(\boldsymbol{u}^{(N)}_{s},  \mathcal T_s (\boldsymbol u_{s+1}^{(N)})\big)$ 
				\EndFor
				\State \Return $\boldsymbol u_0^{(N)}$ as the refined finest-level solution.
			\end{algorithmic}
		\end{algorithm}

		\section{Numerical Results}\label{sec3}
		The proposed SINO framework is implemented in PyTorch and trained using the Adam optimizer. A OneCycleLR schedule is employed to vary the learning rate over the course of training. The experiments are conducted on NVIDIA RTX
		4090 and Titan RTX GPUs. Unless otherwise stated, each model is trained for 500 epochs. 
		
		\subsection{Loss Function}
		
		
		For stationary PDEs and inverse problems such as image super-resolution, the training set is written as $\mathcal D
		=\{(\bar{\boldsymbol f}^{(i)},
		\bar{\boldsymbol u}^{(i)})\}_{i=1}^{M}$, where $\bar{\boldsymbol f}^{(i)}$ denotes the discretized input, such as
		a forcing term, coefficient field, boundary data, or observation, and
		$\bar{\boldsymbol u}^{(i)}$ denotes the corresponding target solution
		or reconstructed field. The parameters are obtained by minimizing
		\begin{equation}
			\label{eq:lossfun2}
			\theta^*
			=
			\operatorname*{arg\,min}_{\theta}
			\mathcal L_{\mathrm{stat}}(\theta),
			\qquad
			\mathcal L_{\mathrm{stat}}(\theta)
			=
			\frac{1}{M}
			\sum_{i=1}^{M}
			\frac{
				\big\|
				\mathcal G_\theta
				(
				\bar{\boldsymbol f}^{(i)}
				)
				-
				\bar{\boldsymbol u}^{(i)}
				\big\|_{L^2(\Omega)}
			}{
				\big\|
				\bar{\boldsymbol u}^{(i)}
				\big\|_{L^2(\Omega)}
			}.
		\end{equation}
		
		For time-dependent problems, the forward operator generally depends explicitly on the time variable \(t\). Accordingly, time is included as an additional input to the learned operator. We first divide the temporal domain \([0,T_{\max}]\) into \(T\) uniform intervals,
		\[
		t_j=j\Delta t,
		\qquad
		\Delta t=\frac{T_{\max}}{T},
		\qquad
		j=0,\ldots,T.
		\]
		Each trajectory therefore consists of the discrete state sequence $\{\bar{\boldsymbol u}_j^{(i)}\}_{j=0}^{T}$, 
		where $i=1,\ldots,M$ indexes the $M$ training trajectories. For one-step prediction, consecutive states are used to construct $MT$ training pairs $\big\{
		(\bar{\boldsymbol u}_j^{(i)},t_j),
		\bar{\boldsymbol u}_{j+1}^{(i)} \big\}$, for $i=1,\ldots,M$ and $j=0,\ldots,T-1$. Here, $\bar{\boldsymbol u}_j^{(i)}$ is the input state at time $t_j$, and $\bar{\boldsymbol u}_{j+1}^{(i)}$ is the corresponding target state at time $t_{j+1}$. The model parameters are determined by minimizing the
		mean relative $L^2$ error:
		\begin{equation}
			\label{eq:lossfun1}
			\theta^*
			=
			\operatorname*{arg\,min}_{\theta}
			\mathcal L_{\mathrm{time}}(\theta),
			\qquad
			\mathcal L_{\mathrm{time}}(\theta)
			=
			\frac{1}{MT}
			\sum_{i=1}^{M}
			\sum_{j=0}^{T-1}
			\frac{
				\big\|
				\mathcal G_\theta
				(
				\bar{\boldsymbol u}_j^{(i)},t_j
				)
				-
				\bar{\boldsymbol u}_{j+1}^{(i)}
				\big\|_{L^2(\Omega)}
			}{
				\big\|
				\bar{\boldsymbol u}_{j+1}^{(i)}
				\big\|_{L^2(\Omega)}
			}.
		\end{equation}
		The relative-error losses in Eq.~\eqref{eq:lossfun1} and Eq.~\eqref{eq:lossfun2} assume that the corresponding target fields have nonzero \(L^2\) norms.

		\subsection{Model Evaluation}\label{sec:model_eva}
		\paragraph{\textbf{Zero-shot resolution generalization}}
		We evaluate the discretization invariance of SINO on the one-dimensional Burgers equation using the standard FNO~\cite{li2020fourier} benchmark setting. The task is to learn the operator mapping the initial condition $\boldsymbol{u}_0$ to the solution $\boldsymbol{u}_1$. We train the FNO and SINO with various iterator steps $N$ at resolution $R=256$ and directly tested at resolutions $R\in\{256, 512, 1024, 2048, 4096, 8192\}$ without retraining. As shown in Figure~\ref{fig:zero_shot_resolution}, SINO generalizes stably across resolutions. Although the model is trained only at $R=256$, its relative error remains nearly stable when evaluated on grids up to $R=8192$. This confirms that the proposed architecture is not tied to the training discretization. Compared with FNO, SINO achieves lower error at all tested resolutions. The improvement becomes more pronounced as the number of iterator steps increases. In particular, SINO with iterator steps $N=32$ achieves the lowest error and outperforms the FNO baseline by approximately one order of magnitude.
		
		\begin{figure}
			\centering
			\includegraphics[width=0.65\linewidth]{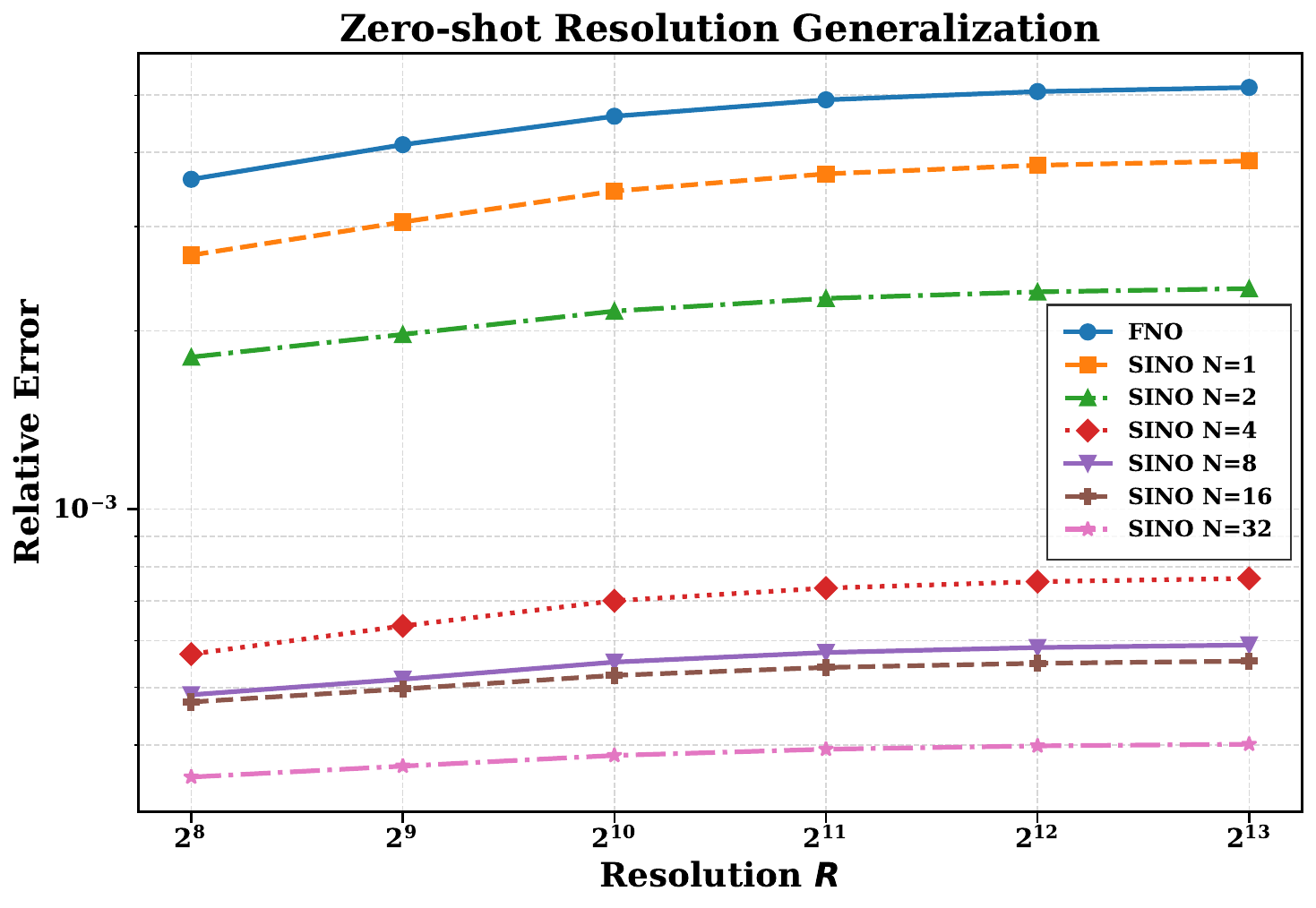}
			\caption{Relative $L^2$ error comparison for zero-shot resolution generalization on the 1D Burgers equation. We trained the FNO and SINO with different iterator steps  $N$ only on low-resolution data and directly evaluated on higher-resolution test sets without any retraining or fine-tuning.}
			\label{fig:zero_shot_resolution}
		\end{figure}
		
		A key observation is that increasing the number of iterator steps does not compromise resolution generalization. The error curves for SINO with larger number of iterations remain stable across test resolutions, indicating that the learned refinement behaves consistently on different grids. This supports our
		function-space interpretation of SINO: the starter and iterator approximate operators between function spaces, while discretization is used only for numerical realization.
		
		\paragraph{\textbf{Hyperparameter selection}}
		We investigate the influence of the architectural hyperparameters on the Darcy-flow benchmark using the standard FNO experimental setting \cite{li2020fourier}. For many PDEs discretized on regular grids, local differential operators produce sparse neighborhood interactions and repeated stencil structures. Convolutional networks provide a suitable inductive bias for representing such local interactions through shared kernels. Although a general inverse or preconditioning operator need not be local or translation equivariant, lightweight CNN blocks can provide an efficient practical parameterization of residual-based corrections. Accordingly, we implement the practical Iterator using CNN blocks and retain an FNO-type spectral block \cite{li2020fourier} for the Starter.

		\begin{figure}[t]
			\centering
				\begin{subfigure}[t]{0.45\linewidth}
					\centering
					\includegraphics[width=\linewidth]{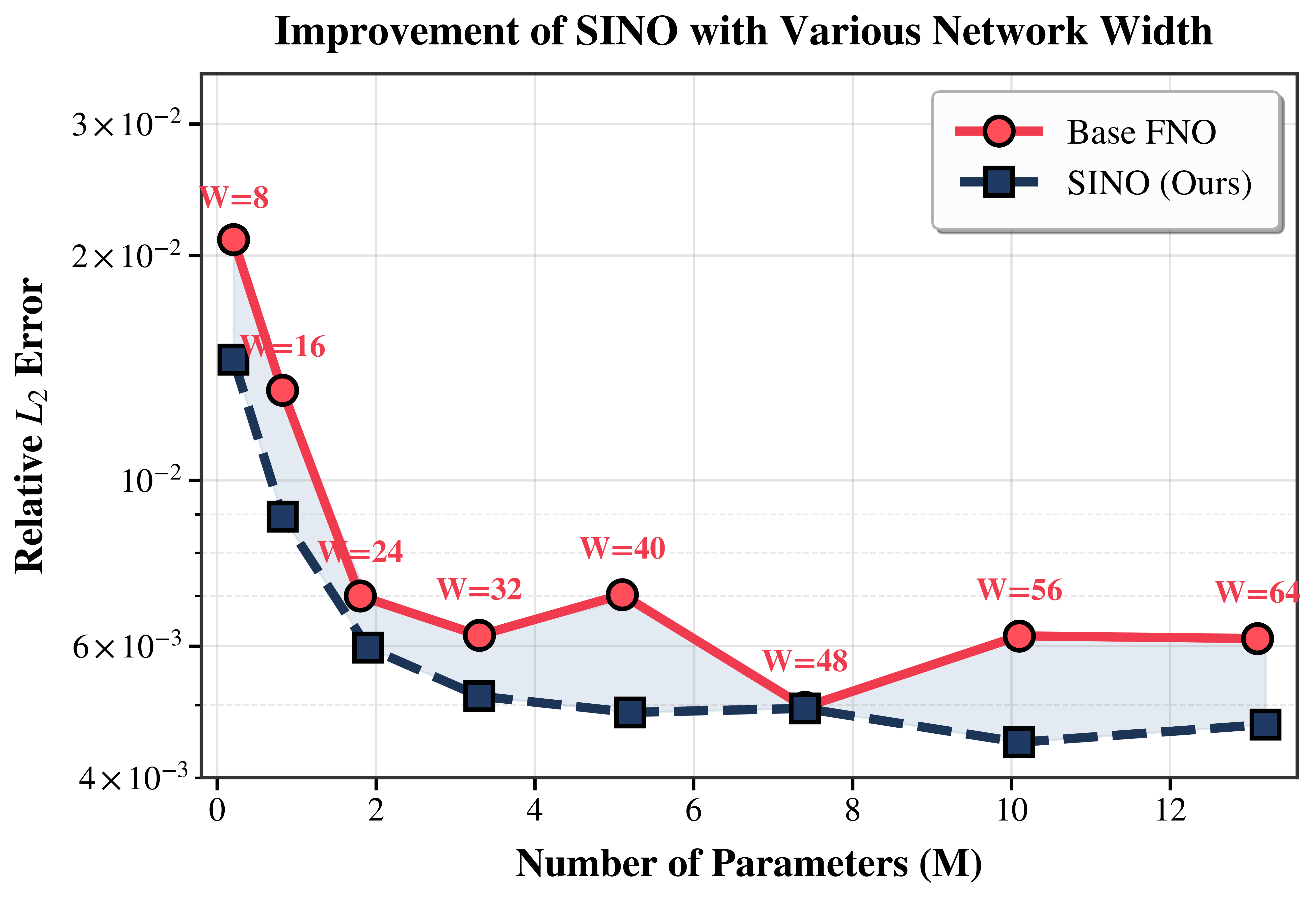}
					\caption{
						Relative error versus model size under fixed Fourier mode
						$K=20$, layers $L=4$, and number of iterations $N=1$. SINO-CNN consistently improves over the
						base FNO across different parameter budgets, indicating better
						parameter efficiency with the learned CNN iterator.
					}
					\label{fig:width_params}
				\end{subfigure}
				\hfill
				\begin{subfigure}[t]{0.48\linewidth}
					\centering
					\includegraphics[width=\linewidth]{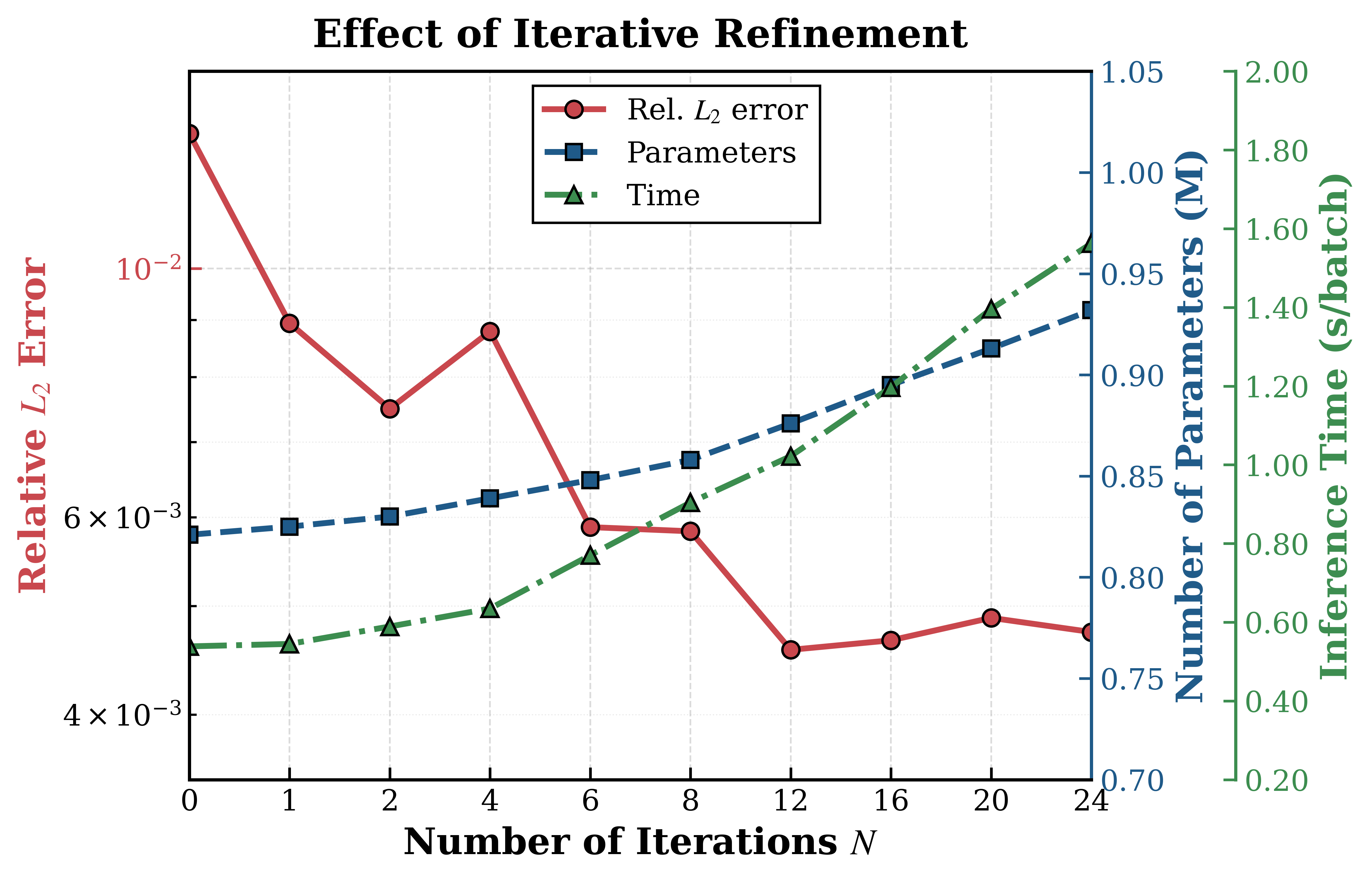}
					\caption{
						Effect of the number of refinement iterator steps $N$ under fixed Fourier
						modes and network width. Increasing $N$ generally reduces the relative
						$L^2$ error, while the number of parameters (M) and inference time (s/batch) grows moderately due to the lightweight CNN-based iterator.
					}
					\label{fig:iter_effect}
				\end{subfigure}
				\caption{
					Hyperparameter analysis on the Darcy-flow benchmark.    \\
					(a) Influence of network width and model size under fixed spectral
					resolution. The proposed SINO achieves lower relative error than the
					base FNO for comparable parameter budgets. 
					(b) Influence of the number of learned refinement iterations. Additional
					iterator steps improve accuracy with only a moderate increase in the number
					of trainable parameters and inference time, demonstrating the effectiveness and parameter
					efficiency of the CNN-based iterative refinement.
				}
				\label{fig:hyperparameter_selection}
			\end{figure}

			We first study the effect of network width at a fixed spectral resolution. Specifically, the number of retained Fourier modes and the number of FNO layers are fixed at $K=20$ and $L=4$, respectively. We compare the base FNO
			with SINO using one Iterator step, i.e., $N=1$, while varying the network width and, consequently, the total number of trainable parameters. As shown in Figure~\ref{fig:hyperparameter_selection}(a), SINO achieves lower relative errors than the base FNO for most of the parameter budgets considered. The improvement is particularly evident in the small- and medium-size regimes, where SINO-CNN provides a more favorable error-parameter trade-off. These results indicate that, on this benchmark, even a single learned refinement step can improve the Starter prediction without requiring a large increase in model size.
			
			We next examine the effect of the number of refinement steps. The Fourier truncation level and network width are held fixed, while the number of Iterator steps $N$ is varied. Figure~\ref{fig:hyperparameter_selection}(b) reports the
			relative $L^2$ error, number of trainable parameters, and inference time per batch. Increasing $N$ generally reduces the relative $L^2$ error, indicating the benefit of successive learned residual corrections. The largest improvement is obtained with the first few refinement steps, while the gain gradually saturates at approximately $N=12$. Because iteration-dependent CNN blocks are used, the number of parameters increases with $N$, but remains moderate over the range considered due to the lightweight design of each block. The inference time increases approximately linearly with $N$, as the refinement steps are evaluated sequentially. For moderate values such as $N\leq10$, the results show a favorable balance between approximation accuracy and computational cost. Thus, the number of Iterator steps provides a controllable accuracy-cost trade-off for the Darcy-flow benchmark.
			
			Finally, the initialization and multiscale results in Figure~\ref{fig:extend_figure2} show that the Starter-based initialization produces lower approximation errors than the zero and random initializations considered in this experiment. The hierarchical multiscale structure further reduces the approximation error on the same benchmark.

			\subsection{Spectral Phenomenon of SINO} \label{subsec:spectral_of_SINO}
			For stationary iterations satisfying the assumptions of
			Theorem \ref{thm:spectral_convergence_operator}(see more details in  ~\ref{sec:Spectral_analysis}), high-frequency error components decay faster than
			low-frequency components. This motivates a spectral Starter
			that reduces the low-frequency initialization error, followed
			by learned residual corrections. We investigate this
			initialization strategy empirically in the practical network.
			
			\begin{figure}[htp!]
				\centering
				\includegraphics[width=0.9\linewidth]{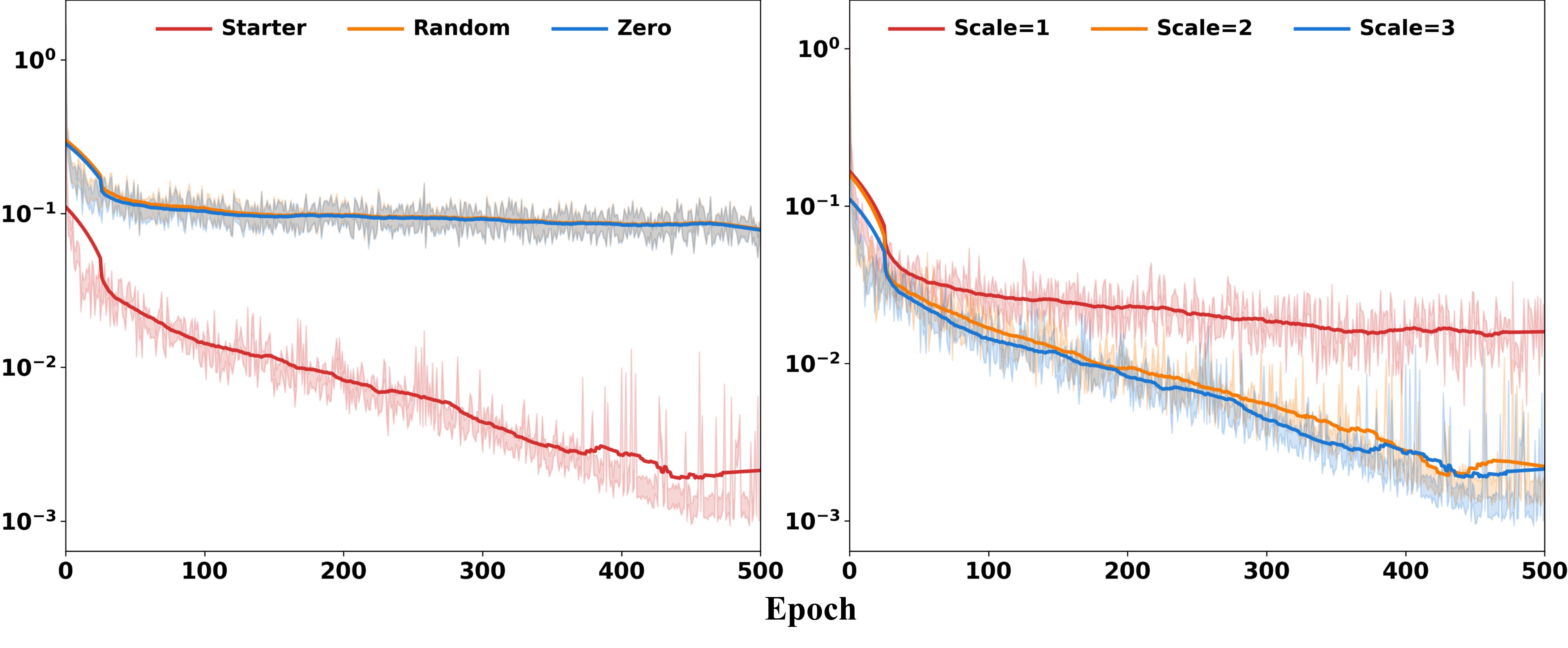}
				\caption{Effect of the starter module and multiscale depth on model convergence for operator learning of the Darcy PDE. The left panel compares training dynamics under three initialization strategies: random initialization, zero initialization, and the proposed starter module. 
					Results indicate that random and zero initializations lead to similar convergence behaviors, whereas the starter module substantially accelerates convergence. 
					The right panel investigates the impact of multiscale depth ($\text{Scale}=1,2,3$) under starter initialization, showing that moderate multiscale depth improves convergence speed, while overly deep architectures may lead to over-parameterization and diminished gains.}
				\label{fig:extend_figure2}
			\end{figure}
			
			\begin{figure}[t]
				\centering
				\includegraphics[width=1\linewidth]{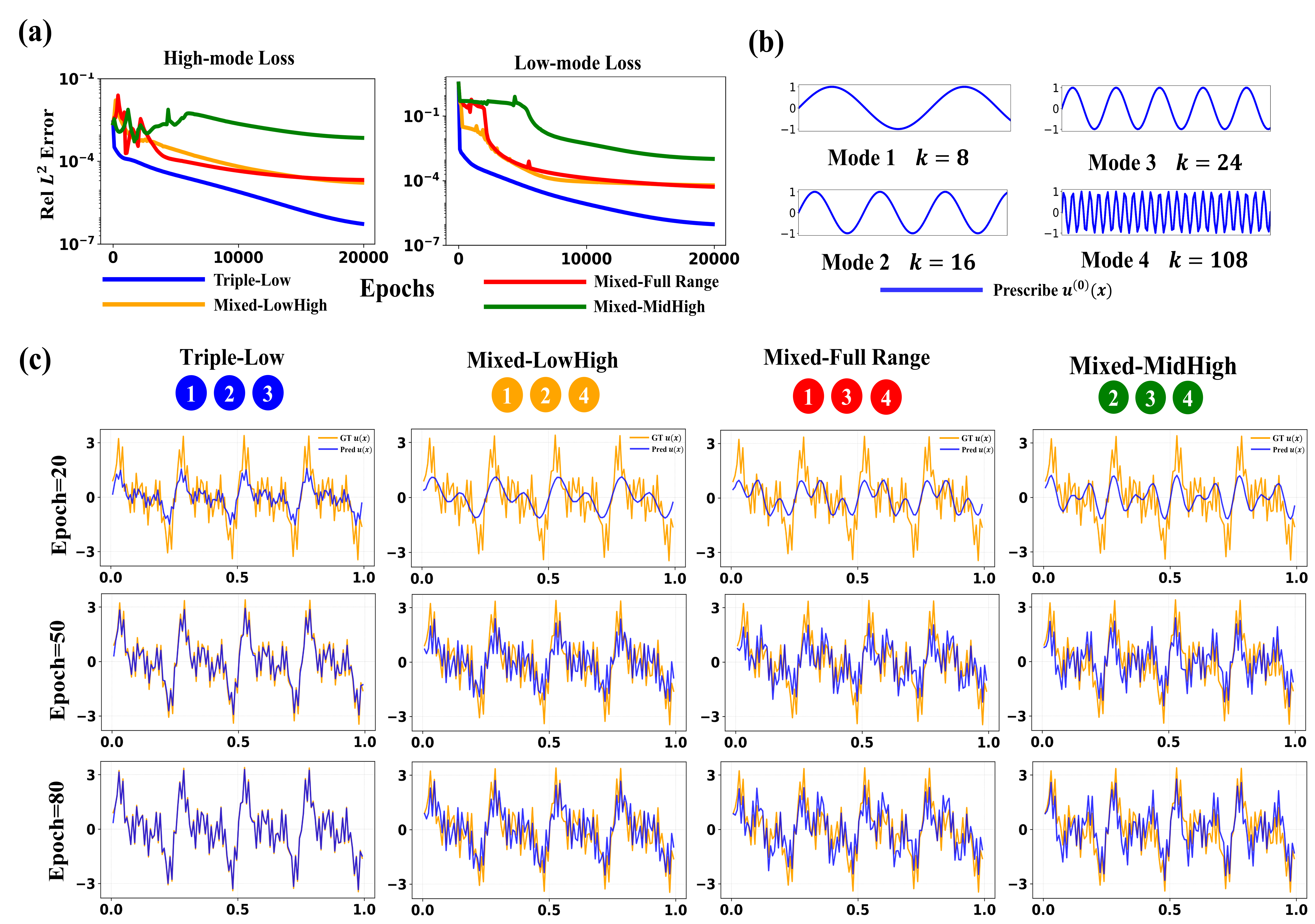}\\
				\vspace{0.3cm}
				\caption{\textbf{Spectral phenomenon of SINO.} (a) Training loss curves for both low- and high-frequency components under different initialization frequency combinations.  (b) Visualization of the four sets of harmonic functions corresponding to the frequency bands selected in the experiments.   (c) Qualitative fitting results at epochs 20, 50, and 80, using three out of the four frequency groups from (b) for initialization.  The specific combinations are indicated with colored markers below each column, for example, “Mixed-MidHigh” refers to modes 2, 3, and 4. The columns in panel~(c), from left to right, correspond to
					the blue, orange, red, and green loss curves in panel~(a).}
				\label{fig:exp1}
			\end{figure}
			
			We consider the one-dimensional Poisson problem
			\[
			\begin{cases}
				-u''(x)=f(x), & x\in(0,1),\\
				u(0)=u(1)=0.
			\end{cases}
			\]
			The reference solution is constructed from prescribed Fourier
			sine modes. Different subsets of these modes are used to
			initialize the network, allowing us to examine how the initial
			frequency content affects training and reconstruction.
			Figure~\ref{fig:exp1} reports the resulting frequency-resolved training
			losses and intermediate predictions.
			
			Initialization with the three lower-frequency modes yields the
			lowest training errors in both frequency bands among the tested
			configurations. The intermediate predictions also show more
			accurate recovery of the remaining oscillatory component.
			These results support the use of low-frequency initialization
			in this experiment. The curves describe optimization across
			training epochs; the fixed-operator iteration analyzed in
			Theorem~\ref{thm:spectral_convergence_operator} provides the motivation for the experiment.

			
			The numerical results are summarized in Figure~\ref{fig:exp1}. As shown in the training loss curves in Figure~\ref{fig:exp1}(a), when the initialization suppresses low-frequency error components, both low- and high-frequency errors diminish rapidly and eventually reach minimal values. In contrast, when the initialization suppresses only high-frequency errors, the network tends to first approximate the dominant low-frequency modes, leading to a loss of existing high-frequency information and ultimately yielding inferior convergence performance.
			
			This observation is further supported by the reconstructions shown in Figure~\ref{fig:exp1}(c). Initializing the iterations with the three lowest frequency modes enables the network to recover the highest frequency features by approximately epoch~20 and nearly complete the fitting process by epoch~80. Conversely, although the alternative initialization strategies introduce high-frequency components at the outset, the network nevertheless begins the approximation from the low-frequency range, causing the high-frequency information to deteriorate around epoch~20 and resulting in slower overall convergence. These findings provide compelling evidence of the importance of initializing with reduced low-frequency error in achieving efficient and accurate iterative refinement.

			\subsection{Time Bundling Strategy}  
			We use a manufactured solution of the one-dimensional heat
			equation to study the effect of input and output window lengths.
			Let $\mathbb K\subset\{2,\ldots,24\}$ contain between one and
			six randomly selected frequencies, and define
			\[
			u(x,t)
			=
			e^{-t}\sum_{k\in\mathbb K}\sin(2\pi kx).
			\]
			Then $u$ satisfies
			\[
			\begin{cases}
				\partial_t u=\partial_{xx}u+f(x,t),
				& (x,t)\in(0,1)\times(0,0.1],\\
				u(0,t)=u(1,t)=0,
				& t\in[0,0.1],\\
				u(x,0)=\displaystyle\sum_{k\in\mathbb K}\sin(2\pi kx),
				& x\in[0,1],
			\end{cases}
			\]
			with
			\[
			f(x,t)
			=
			e^{-t}\sum_{k\in\mathbb K}
			\bigl((2\pi k)^2-1\bigr)\sin(2\pi kx).
			\]
			The fields are sampled at $n_x=1024$ spatial points and
			$n_t=100$ time points. More than $20{,}000$ frequency
			combinations are used for training.
			This construction isolates the effect of temporal windowing
			through a common exponential decay across spatial modes.

			\begin{figure}
				\centering
				\includegraphics[width=1\linewidth]{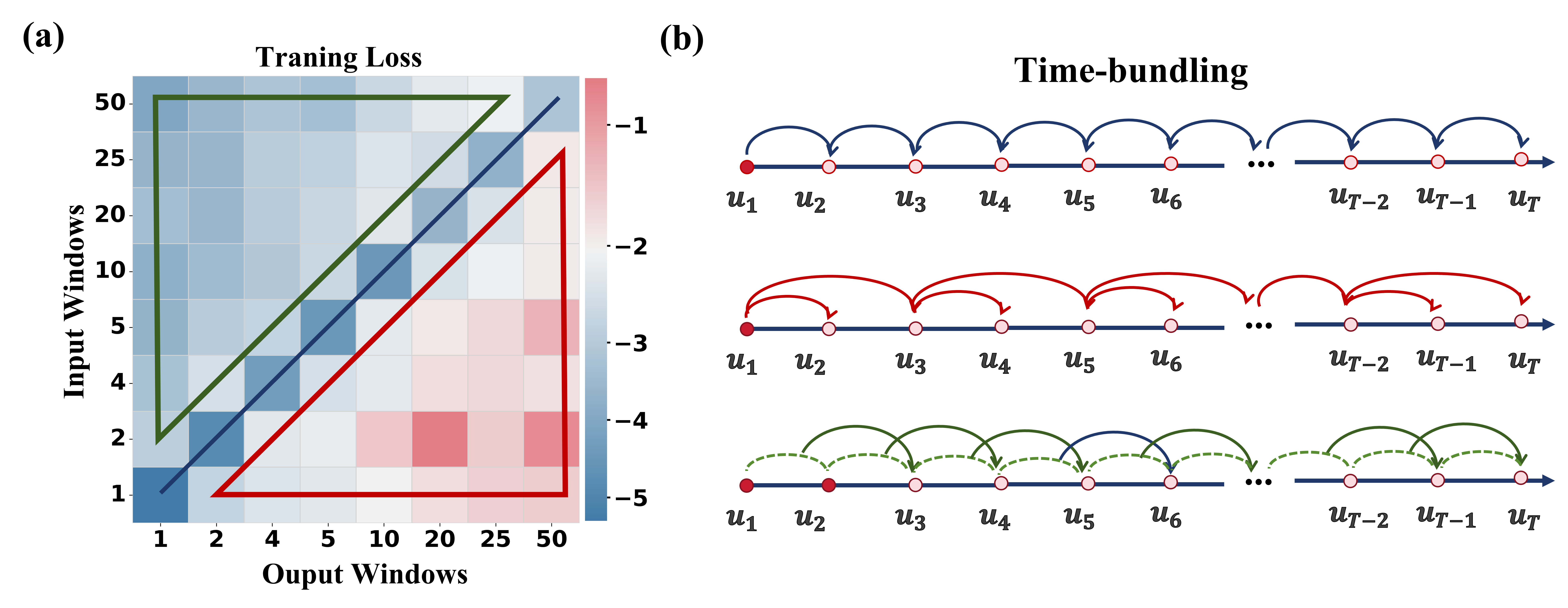}
				\caption{(a) Heatmap of training losses with respect to different input and output sequence lengths $(L_{\text{in}}, L_{\text{out}})$. 
					Three characteristic regions can be observed: the upper-left green triangle corresponds to $L_{\text{in}} > L_{\text{out}}$, 
					representing the case where the network predicts a shorter horizon than the available input (cf. third row in (b)); 
					the lower-right red triangle indicates $L_{\text{in}} < L_{\text{out}}$, where the network is required to extrapolate beyond the input sequence (cf. second row in (b)); 
					and the blue diagonal line represents the matched case $L_{\text{in}} = L_{\text{out}}$, i.e., one-to-one temporal mapping (cf. first row in (b)). 
					Panel (b) illustrates the corresponding temporal dependency structures under these three regimes.%
				}
				\label{fig:timebundling}
			\end{figure}
			Figure~\ref{fig:timebundling}(a) compares training losses for different
			$(L_{\mathrm{in}},L_{\mathrm{out}})$.
			Matched input and output window lengths produce lower losses
			in this manufactured-solution experiment.
			The result characterizes the tested temporal-windowing setup;
			the preferred window lengths may vary with the dynamics,
			sampling interval, and forecasting task.
			
			
			\subsection{Results on Time-evolution Equations} \label{time_results}
			Time-dependent partial differential equations Eq.~\eqref{pde} are essential mathematical models for describing dynamic processes in physical systems, and their numerical solutions form the foundation of forward problem studies. We investigate the performance of SINO on several time-evolution PDEs, including the incompressible Navier-Stokes equation, acoustic wave equation, shallow-water equation, etc. Specifically, we conduct comparative experiments on three representative classes of evolution equations, benchmarking them against three widely used neural operator baselines: FNO \cite{li2020fourier}, DeepONet \cite{LuDeepOnet}, and MgNO \cite{he2024mgno}.
			The relative errors, computed according to Eq. \eqref{eq:lossfun1}, obtained upon convergence of the training process are summarized in Table \ref{tab:error_comparison}. This table provides a quantitative comparison of the one-step prediction errors of DeepONet, FNO, MgNO, and SINO on the test datasets, clearly demonstrating the superior performance of our method over the baseline approaches.

			\textbf{Incompressible Navier-Stokes equation.} The governing equations for an incompressible viscous fluid flow in vorticity form on the unit torus subjected to an external force field \( \boldsymbol{f}(\mathbf{x},t) \) are expressed as follows:
			\begin{equation}\label{NS}
				\partial_t \boldsymbol{\omega}(\mathbf{x},t) + \boldsymbol{u}(\mathbf{x},t)\cdot \nabla \boldsymbol{\omega}(\mathbf{x},t) 
				= \nu \Delta \boldsymbol{\omega}(\mathbf{x},t) + \boldsymbol{f}(\mathbf{x},t), \quad (\mathbf{x}, t) \in \Omega \times (0, T]
			\end{equation}
			subject to the initial condition \(\boldsymbol{\omega}(\mathbf{x},0) = \boldsymbol{\omega}_0(\mathbf{x})\), where \(\omega(\mathbf{x},t)\) is the scalar vorticity field, with \(\boldsymbol{\omega} = \nabla \times \boldsymbol{u}\) in 2D, and \(\boldsymbol{u}(\mathbf{x},t)\) is the velocity field that satisfies the incompressibility condition \(\nabla \cdot \boldsymbol{u} = 0\). The parameter \(\nu\) is the viscosity coefficient, and \(\boldsymbol{f}(\mathbf{x},t)\) represents the external forcing in the vorticity equation. 
			We train the SINO model to solve Eq. \eqref{NS}, which learns the mapping \(\mathcal{G}\) between vorticity fields \(\boldsymbol{\omega}(\mathbf{x}, t)\) at two consecutive time steps during training:
			\[
			\mathcal{G}: \boldsymbol{\omega}(\mathbf{x}, t) \mapsto \boldsymbol{\omega}(\mathbf{x}, t+\Delta t).
			\]
			We consider two different viscous incompressible flows with viscosities $\nu = 10^{-3}$ and $\nu = 10^{-4}$, respectively. 
			Table~\ref{tab:error_comparison} reports the one-step prediction errors for both
			viscosity settings. SINO achieves the lowest errors among
			the compared models. Figure~\ref{fig:evolutionPDE} further evaluates autoregressive
			prediction over 50 steps, where SINO maintains lower errors
			throughout the tested horizons.
			
			The lower-viscosity case exhibits weaker diffusion and more
			complex small-scale structures, and all methods incur larger
			prediction errors. SINO retains an accuracy advantage in this
			more demanding regime.
			
			\begin{table}[t]
				\centering
				\caption{One-step Error Comparison between our SINO and other established neural operators. }
				\label{tab:error_comparison}
				\begin{tabular}{c|cc|cc|c}
					\toprule
					\multirow{2}{*}{\textbf{Model}} & \multicolumn{2}{c|}{\textbf{Navier-Stokes}} & \multicolumn{2}{c|} {\textbf{Acoustic-Wave}}  & \multirow{2}{*}{\textbf{Shallow-Water}} \\
					& \(\nu = 1e-3\) & \(\nu = 1e-4\) & source variable & parameter variable &     \\
					\midrule
					DeepONet   & 1.720e-01  & 3.823e-01  & 2.046e-01 & 6.040e-01 & 3.000e-03 \\
					\midrule
					MgNO        & 2.200e-03    & 4.950e-02   & 1.100e-03 & 4.220e-02 & 9.000e-04 \\
					\midrule
					FNO        & 8.000e-04    & 7.000e-02   & 1.800e-03 & 3.730e-02 & 1.300e-03  \\
					\midrule
					\cellcolor{gray!30} \textbf{SINO}    & \cellcolor{gray!30}\textbf{3.000e-04}    & \cellcolor{gray!30} \textbf{8.400e-03}    & \cellcolor{gray!30}\textbf{5.000e-04}  & 
					\cellcolor{gray!30}\textbf{1.940e-02} & \cellcolor{gray!30}\textbf{6.000e-04}\\
					\bottomrule
				\end{tabular}
			\end{table}
			
			To approximate the solution to the equation, the learned operator is applied autoregressively to the initial condition \(\boldsymbol{\omega}_0(\mathbf{x})\) for iterative solving, a process that demands the neural operator possess strong generalization capabilities. The left panel of Figure \ref{fig:evolutionPDE}(b) shows the error evolution curves over 50 autoregressive time steps. As shown in Figure~\ref{fig:evolutionPDE}(b), when the viscosity coefficient is relatively large and the dynamics remain smooth, spectral approaches maintain a stable trend, whereas larger viscosity leads to slower accumulation of prediction error. 
			Our proposed SINO incorporates spectral learning in the initialization step and further refines high-frequency details through parameterized iterative updates, consistently achieving lower single-step error than FNO and offering improved stability in long-term forecasting. Furthermore, since the condition number of the discretized linear system varies with the viscosity coefficient, the accuracy of neural operators differs across scenarios. A smaller viscosity coefficient leads to a more ill-conditioned system, thereby increasing the difficulty of solving the equations. As SINO is designed based on iterative schemes, its efficiency also declines under such conditions, similar to classical iterative solvers.
			
			\begin{figure}
				\centering
				\includegraphics[width=0.95\linewidth]{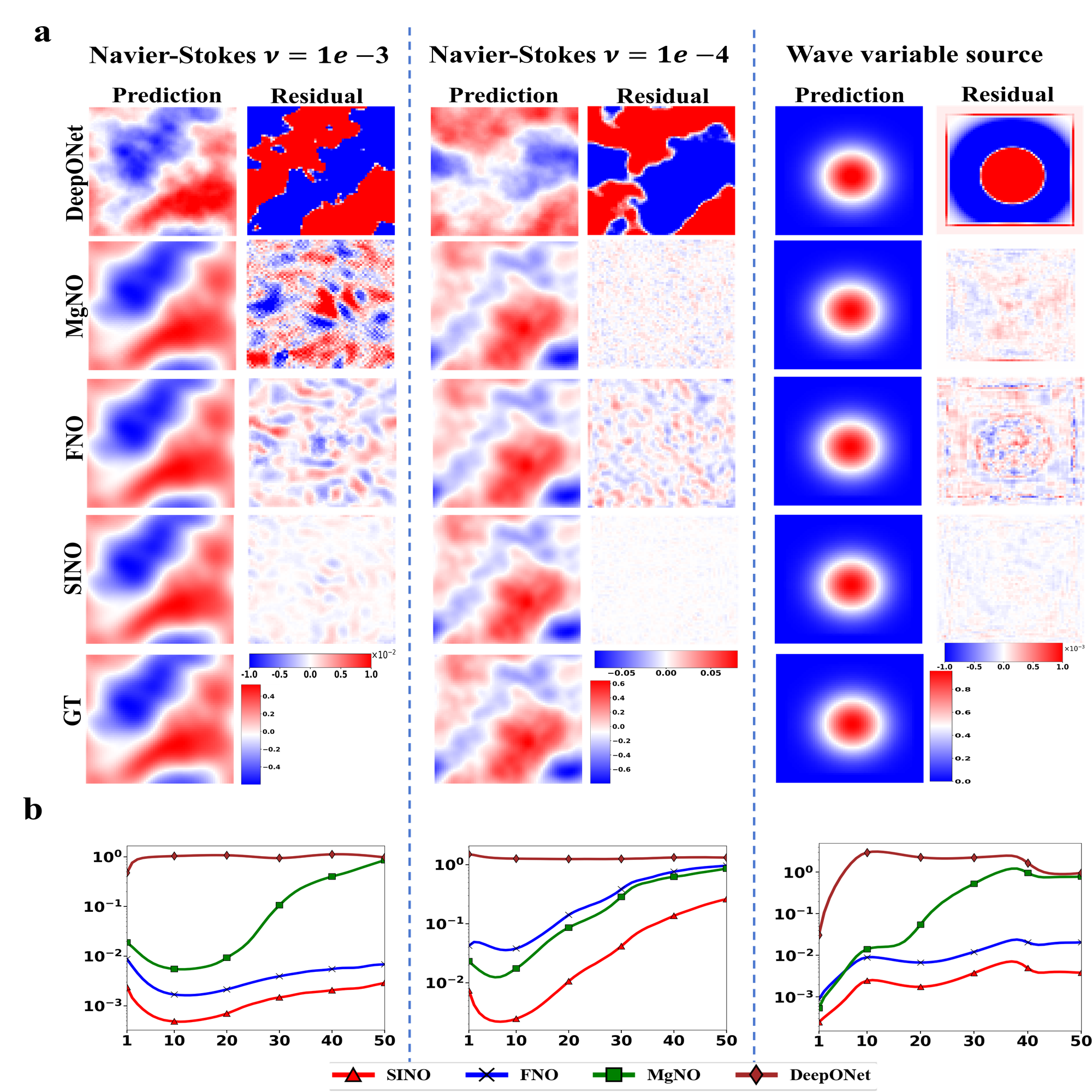}
				\vspace{0.1cm}
				\caption{\textbf{Results on Navier-Stokes equation and Wave equation.} The first and second columns show predicted solutions (first column) and corresponding error maps (second column) for the 2D Navier–Stokes equation with viscosity $\nu=1e-3$, generated by different operator learning methods (one per row). The third and fourth columns display analogous results under reduced viscosity $\nu=1e-4$. The fifth and sixth columns present predictions and residuals for the 2D wave equation with varying source terms. The last row in each block corresponds to the ground truth. Our method consistently exhibits the lowest prediction error and clearest reconstruction across varying physical regimes. The bottom row shows the evolution of relative $L^2$
					errors over 50 autoregressive inference steps for the three test scenarios, highlighting the superior long-term accuracy and stability of our model compared to baselines.}
				\label{fig:evolutionPDE}
			\end{figure}

			\textbf{Acoustic Wave Equation.}
			The acoustic wave equation is a fundamental model describing the propagation of sound waves in a medium. It is widely used in geophysics, medical imaging, and non-destructive testing for simulating acoustic wave fields. The two-dimensional acoustic wave equation with a source term is expressed as
			\begin{equation*}
				\frac{\partial^2 \boldsymbol{u}(\mathbf{x},t)}{\partial t^2} 
				= \frac{1}{c^2(\mathbf{x})} 
				\frac{\partial^2 \boldsymbol{u}(\mathbf{x},t)}{\partial \mathbf{x}^2}  
				+ \boldsymbol{f}(\mathbf{x},t), \quad \mathbf{x} \in \Omega.
			\end{equation*}
			with the given boundary conditions
			\begin{equation*}
				\left\{
				\begin{array}{@{}l@{\quad}l@{}}
					\displaystyle
					\boldsymbol{u}(\mathbf{x}, 0) = u_0 (\mathbf{x}), & \mathbf{x} \in  \Omega,
					\\[1.6ex]
					\displaystyle
					\frac{\partial \boldsymbol{u}(\mathbf{x}, 0)}{\partial t} =0 & \mathbf{x} \in \Omega,
				\end{array}
				\right.
			\end{equation*}
			where \(\boldsymbol{u}(\mathbf{x},t)\) denotes the acoustic pressure field, \(c(\mathbf{x})\) represents the wave speed in the medium, and \(\boldsymbol{f}(\mathbf{x},t)\) is the external source term. The initial condition is set as \(u_0 (\mathbf{x}) = \boldsymbol{f}(\mathbf{x}, 0)\) corresponding to a quiescent medium prior to excitation. Absorbing boundary conditions are applied to prevent reflections from the edges of the computational domain.
			
			In scenarios where the wave speed distribution \(c(\mathbf{x})\) of the medium is fixed and the source location varies, the steadiness of the velocity field enables us to model the transition between successive time steps given a specific source term. The initial condition inherently encodes information about the source function, allowing us to learn a mapping \(\mathcal{G}\) between consecutive temporal states
			\[\mathcal{G}: \boldsymbol{u}(\mathbf{x},t) \mapsto \boldsymbol{u}(\mathbf{x},t+\Delta t).\] 
			As illustrated in Table \ref{tab:error_comparison} and Figure \ref{fig:evolutionPDE}, our results are consistent with the findings for the Navier-Stokes equations, with our SINO demonstrating the highest precision in both single-step predictions and iterative rollouts. 
			{The starter-iterator architecture effectively serves as a robust framework for operator learning, demonstrating competitive performance across different simulation scenarios.}

			In scenarios where the wave speed distribution \(c(\mathbf{x})\) varies while the source location remains fixed, 
			the learning task becomes a more challenging mapping 
			\[\mathcal{G}: \big(\boldsymbol{u}(\mathbf{x},t); c(\mathbf{x}) \big) \mapsto \boldsymbol{u}(\mathbf{x},t+\Delta t).\] 
			The velocity distribution \(c(\mathbf{x})\) is obtained from simulated breast tissue slice data 
			\cite{zeng2025openbreastus}, and the training dataset is generated using a 500~kHz ultrahigh-frequency wave 
			(see ~\ref{secATeq} for additional details on data generation). 
			{We explore the application of neural operators to the simulation of ultrahigh-frequency acoustic waves, demonstrating their potential in this challenging regime.}
			The one-step prediction results are summarized in Table~\ref{tab:error_comparison}, 
			where our proposed SINO consistently achieves the highest accuracy among all baselines. 
			Furthermore, additional visualizations (as shown in Figure~\ref{fig:High_wave}) illustrate that our approach 
			remains stable over long-term propagation, producing physically meaningful pressure fields. 
			In contrast, conventional operator learning methods tend to accumulate errors rapidly, 
			resulting in predictions that lose physical interpretability. 
			This highlights the superior temporal stability of our method.
			
			{\textbf{Longer horizon inference.}
				To further validate the temporal stability of our method, we conduct an additional experiment on the
				one-dimensional linear-advection equation from PDEBench~\cite{PDEbench}. The governing equation is
				given by
				\begin{equation}
					\partial_t u(t,x) + \beta \partial_x u(t,x) = 0,
					\qquad x\in(0,1),\quad t\in(0,2],
					\label{eq:advection}
				\end{equation}
				with the initial condition
				\begin{equation}
					u(0,x)=u_0(x), \qquad x\in(0,1),
				\end{equation}
				where $\beta$ denotes a constant advection speed. The problem is considered
				under periodic boundary conditions. For this equation, the solution corresponds
				to a translation of the initial profile,
				\[
				u(t,x)=u_0(x-\beta t),
				\]
				which makes it a suitable benchmark for evaluating whether a learned model can
				preserve transported structures over long temporal horizons.
				We train the SINO with 1000 trajectories to compare with other competitive models and report the relative $L_2$ error along a $100$-steps prediction horizon.
				
				As shown in Figure ~\ref{fig:long_horizon_advection_diffusion}, the prediction error
				increases with time for all methods, which is expected for autoregressive
				rollouts due to accumulated numerical and modeling errors. Nevertheless, SINO
				consistently achieves the lowest relative $L_2$ error over the entire horizon.
				Compared with FNO, SINO exhibits a significantly slower error growth rate,
				indicating better stability in long-time propagation. Compared with MgNO, SINO
				also maintains a clear advantage across most rollout steps. These results
				suggest that the proposed starter-iterator structure is effective not only for
				short-term operator approximation but also for mitigating error accumulation in
				long-horizon time-dependent PDE prediction.}

			\textbf{Ablation studies.}
			The shallow-water equations, derived as a depth-averaged approximation of the Navier-Stokes equations, offer a traditional framework for modeling large-scale geophysical flows, including ocean currents, tidal waves, and atmospheric circulation. These equations address the horizontal motion of a thin fluid layer, assuming the horizontal dimensions significantly exceed the vertical depth. In two spatial dimensions with external forces, the system is expressed as follows:
			\begin{equation*}
				\left\{
				\begin{array}{@{}l@{\quad}l@{}}
					\displaystyle
					\frac{\partial (\boldsymbol{h}\boldsymbol{u})}{\partial t} 
					+ \frac{\partial}{\partial x}\!\left( \boldsymbol{h}\boldsymbol{u}^2 + \tfrac{1}{2} g \boldsymbol{h}^2 \right) 
					+ \frac{\partial (\boldsymbol{h}\boldsymbol{u}\boldsymbol{v})}{\partial y}
					& \displaystyle = - g \boldsymbol{h}\frac{\partial \boldsymbol{b}}{\partial x} + \boldsymbol{F}_x(x,y,t),
					\\[1.2ex]
					\displaystyle
					\frac{\partial (\boldsymbol{h}\boldsymbol{v})}{\partial t} 
					+ \frac{\partial (\boldsymbol{h}\boldsymbol{u}\boldsymbol{v})}{\partial x} 
					+ \frac{\partial}{\partial y}\!\left( \boldsymbol{h}\boldsymbol{v}^2 + \tfrac{1}{2} g \boldsymbol{h}^2 \right) 
					& \displaystyle =  - g \boldsymbol{h}\frac{\partial \boldsymbol{b}}{\partial y} + \boldsymbol{F}_y(x,y,t),
				\end{array}
				\right.
			\end{equation*}
			subject to the conservation law
			\begin{equation*}
				\frac{\partial \boldsymbol{h}}{\partial t} + \frac{\partial (\boldsymbol{h}\boldsymbol{u})}{\partial x} + \frac{\partial (\boldsymbol{h}\boldsymbol{v})}{\partial y}  = 0,
			\end{equation*}
			where \( \boldsymbol{h}(x,y,t) \) is the fluid depth, and \( \boldsymbol{u}(x,y,t) \) and \( \boldsymbol{v}(x,y,t) \) are the horizontal velocity components in the \( x \) and \( y \) directions, respectively. The constant \( g \) denotes the gravitational acceleration, and \( \boldsymbol{b} \) describes a spatially varying bathymetry. The terms \( \boldsymbol{F}_x(x,y,t) \) and \( \boldsymbol{F}_y(x,y,t) \) represent external forces in the two horizontal directions. Our aim is to learn the mappings between consecutive time steps for the velocity variables $u$ and $v$, respectively
			\begin{equation*}
				\mathcal{G}: \boldsymbol{u}(x, y, t) \mapsto \boldsymbol{u}(x,y,t+\Delta t) \quad \text{and} \quad  \mathcal{G}: \boldsymbol{v}(x, y, t) \mapsto \boldsymbol{v}(x,y,t+\Delta t).
			\end{equation*}
			We further perform ablation studies and extrapolation evaluations using a simplified format of the Navier-Stokes equations: the shallow-water equations (SWE), exploring the influence of three strategies on the performance of our model. Specifically, we examine three key components: the starter module $S$, the iterator module $I$, and the multi-scale structure $M$. The ablation results are summarized in Table~\ref{Ablation}, where we report the relative $L^2$ error of the model predictions at the 10th, 20th, 30th, 40th, and 50th time steps. 
			
			\begin{table*}[htbp]
				\centering
				\caption{Ablation studies on the shallow-water equation, where $S$ denotes the starter module, $I$ denotes the iterator module, and $M$ denotes the multi-scale framework. We report the relative $L^2$ error at the 10th, 20th, 30th, 40th, and 50th inference steps.}
				\label{Ablation}
				\begin{tabular}{ccc|ccccc}
					\toprule
					$S$ & $I$ & $M$  & \textbf{T=10} & \textbf{T=20} & \textbf{T=30} & \textbf{T=40} & \textbf{T=50}\\
					\midrule
					\checkmark & \quad   & \quad   &  9.094e-04   & 2.097e-03  & 8.637e-03 & 2.368e-02 & 5.259e-02     \\
					\checkmark  & \quad & \checkmark    & 1.878e-03  & 2.679e-03  & 3.600e-03 & 4.565e-03 & 5.848e-03  \\
					\quad & \checkmark   & \checkmark    &  2.146e-03   & 3.300e-03 & 4.662e-03 & 7.663e-03 & 1.612e-02 \\
					\checkmark  & \checkmark & \quad      &  4.087e-04   & 1.263e-03 & 2.726e-03 & 4.511e-03 & 6.801e-03   \\
					\cellcolor{gray!30}\checkmark & \cellcolor{gray!30}\checkmark & \cellcolor{gray!30}\checkmark  &  \cellcolor{gray!30}\textbf{2.634e-04} & \cellcolor{gray!30}\textbf{4.471e-04} & \cellcolor{gray!30}\textbf{5.775e-04} & \cellcolor{gray!30}\textbf{6.551e-04} & \cellcolor{gray!30}\textbf{7.593e-04} \\
					\bottomrule
				\end{tabular}
			\end{table*}
			
			
			
			Table~\ref{Ablation} shows that the complete Starter-Iterator architecture with multiscale correction achieves the lowest error at every reported rollout step. At step 50, its relative error is
			$7.593\times10^{-4}$, compared with $6.801\times10^{-3}$
			without the multiscale structure and $1.612\times10^{-2}$
			without the Starter. These results support the complementary
			roles of spectral initialization, residual refinement, and
			cross-scale correction on this benchmark.

			\subsection{Application on Super-Resolution Imaging}  \label{exp:sr}
			Super-resolution microscopy (SRM) stands as a pivotal advancement in life sciences imaging, empowering researchers to transcend the optical diffraction limit and explore subcellular structures with unprecedented clarity. By surpassing the conventional Rayleigh criterion, SRM techniques such as stochastic optical reconstruction microscopy (STORM) \cite{rust2006sub}, photoactivated localization microscopy (PALM) \cite{betzig2006imaging}, and structured illumination microscopy (SIM) \cite{li2015extended} offer profound insights into biological processes at the nanoscale. Theoretically, the fluorescence microscopy imaging process can be modeled as a spatial convolution between the high-resolution distribution of fluorophores and the system’s point spread function (PSF), with measurement noise posing a significant challenge to the fidelity of the acquired data.  
			The degradation process of high-resolution fluorophores can be described by the integral Eq. \eqref{integral}.  
			The high-resolution fluorophore distribution is represented in vector form as \(\mathbf{x}_{\mathrm{sr}} \in \mathbb{R}^n\), and the observed low-resolution data is denoted as \(\mathbf{x}_{\mathrm{lr}} \in \mathbb{R}^m\). Due to the reduction in resolution, it generally holds that \(m < n\). The corresponding image reconstruction problem can be formulated as the following linear inverse problem:
			\begin{equation}\label{sim-back}
				\mathbf{x}_{\mathrm{lr}} = A \mathbf{x}_{\mathrm{sr}} + \boldsymbol{\eta},
			\end{equation}
			where \(A \in \mathbb{R}^{m \times n}\) is the imaging operator, each row of which is constructed by shifting and downsampling the point spread function (PSF). Since the system is underdetermined, solving Eq. \eqref{sim-back} is an ill-posed problem, which can be addressed using operator learning method by
			\begin{equation*}
				\mathcal{G}: \mathbf{x}_{\mathrm{lr}} \mapsto \mathbf{x}_{\mathrm{sr}}.
			\end{equation*}
			We utilize the publicly available BioSR dataset \cite{qiao2021evaluation} for super-resolution tasks, which comprises over 2,200 paired low-resolution and high-resolution images acquired through a multimodal structured illumination microscopy (SIM) system. The dataset includes four types of intracellular structures with increasing structural complexity: clathrin-coated pits (CCPs), microtubules (MTs), F-actin filaments, and endoplasmic reticulum (ER). As the resolution becomes higher, the reconstruction challenges also increase. For each structural type, approximately 50 regions of interest (ROIs) were captured under nine different excitation intensities. The corresponding raw SIM images were processed to generate paired diffraction-limited wide-field (WF) images and high-resolution SIM reconstructions. To ensure data quality, samples of F-actin, MTs, and CCPs were imaged from fixed cells, whereas ER data were acquired from live cells due to fixation-induced artifacts.
			
			\begin{figure}[htp!]
				\centering
				\includegraphics[width = 0.9\linewidth]{Figure5v2.png}\\
				\vspace{0.1cm}
				\caption{{\textbf{Qualitative comparison of $\times 2$ super-resolution reconstructions across four datasets.}} 
					(\textbf{a}) Reconstructed images and absolute error maps. For each dataset, the first row shows the widefield (WF) input followed by the reconstructions obtained using the compared methods. The second row shows the ground-truth SIM (GT-SIM) image followed by the absolute differences between each reconstruction and GT-SIM. Arrows indicate the locations and directions of the intensity profiles.
					(\textbf{b}) Intensity profiles along the paths indicated in (\textbf{a}), comparing the reconstructed and reference intensities to assess local agreement in peak amplitudes, edge transitions, and contrast.
				}
				\label{fig4}
			\end{figure}

			{  \begin{table}[t]
					\centering
					\caption{
						Quantitative comparison of different methods on the CCPs, ER, MTs, 
						and F-actin datasets for the super-resolution task.
						The evaluation metrics include MSE, PSNR, and MS-SSIM.
						Lower MSE and higher PSNR and MS-SSIM indicate better performance.
						The results of the proposed SINO are highlighted in bold with a gray background.
					}
					\label{tab:quantitative_comparison_biosr}
					\setlength{\tabcolsep}{10pt}
					\renewcommand{\arraystretch}{1.15}
					\begin{tabular}{llccc}
						\toprule
						\textbf{Dataset} 
						& \textbf{Method}
						& \textbf{MSE} $\downarrow$
						& \textbf{PSNR (dB)} $\uparrow$
						& \textbf{MS-SSIM} $\uparrow$ \\
						\midrule
						
						\multirow{4}{*}{\textbf{CCPs}}
						& FNO    
						& $3.0689\times10^{-4}$ 
						& 35.4348 
						& 0.9829 \\
						
						& RCAN   
						& $2.3169\times10^{-4}$ 
						& 36.6238 
						& 0.9895 \\

						& DFCAN  
						& $2.0129\times10^{-4}$ 
						& 37.2426 
						& 0.9911 \\
						
						\rowcolor{gray!18}
						& \textbf{SINO} 
						& $\mathbf{1.4669\times10^{-4}}$ 
						& \textbf{38.6213} 
						& \textbf{0.9940} \\
						
						\midrule
						
						\multirow{4}{*}{\textbf{ER}}
						& FNO    
						& $6.3244\times10^{-4}$ 
						& 32.6740 
						& 0.9808 \\
						
						& RCAN   
						& $5.7385\times10^{-4}$ 
						& 33.3163 
						& 0.9839 \\
						
						& DFCAN  
						& $6.0441\times10^{-4}$ 
						& 33.3237 
						& 0.9834 \\
						
						\rowcolor{gray!18}
						& \textbf{SINO} 
						& $\mathbf{5.6906\times10^{-4}}$ 
						& \textbf{33.4968} 
						& \textbf{0.9843} \\
						
						\midrule
						
						\multirow{4}{*}{\textbf{MTs}}
						& FNO    
						& $5.5768\times10^{-4}$ 
						& 32.7990 
						& 0.9713 \\
						
						& RCAN   
						& $4.5256\times10^{-4}$ 
						& 33.7187 
						& 0.9777 \\
						
						& DFCAN  
						& $4.3483\times10^{-4}$ 
						& 33.8515 
						& 0.9776 \\
						
						\rowcolor{gray!18}
						& \textbf{SINO} 
						& $\mathbf{4.3084\times10^{-4}}$ 
						& \textbf{34.0082} 
						& \textbf{0.9782} \\
						
						\midrule
						
						\multirow{4}{*}{\textbf{F-actin}}
						& FNO    
						& $6.4144\times10^{-4}$ 
						& 32.5616 
						& 0.9590 \\
						
						& RCAN   
						& $5.7499\times10^{-4}$ 
						& 33.0470 
						& 0.9669 \\
						
						& DFCAN  
						& $5.4061\times10^{-4}$ 
						& 33.2597 
						& 0.9668 \\
						
						\rowcolor{gray!18}
						& \textbf{SINO} 
						& $\mathbf{4.8392\times10^{-4}}$ 
						& \textbf{33.6241} 
						& \textbf{0.9681} \\
						
						\bottomrule
					\end{tabular}
			\end{table}}
			
			We evaluate SINO against the Fourier neural operator (FNO)~\cite{li2020fourier} and two established models for structured illumination microscopy (SIM) super-resolution: the residual channel attention network (RCAN)~\cite{zhang2018image} and the deep Fourier channel attention network (DFCAN)~\cite{qiao2021evaluation}. Figure~\ref{fig4} compares the reconstructed images and residual maps across four datasets. SINO shows close agreement with the ground-truth SIM (GT-SIM) images, with smaller residuals and fewer structure-dependent artifacts than the competing methods. The intensity profiles in Figure~\ref{fig4}(b) further illustrate this agreement: SINO generally reproduces the reference peak amplitudes, edge transitions, and local contrast more accurately.
			
			Table~\ref{tab:quantitative_comparison_biosr} summarizes the corresponding reconstruction errors and image-quality metrics. SINO attains the lowest MSE and the highest PSNR and MS-SSIM among the methods considered on each of the four datasets. On CCPs, it achieves a PSNR of 38.62~dB, exceeding DFCAN and RCAN by 1.38 and 2.00~dB, respectively, with an MSE of $1.4669\times10^{-4}$. Improvements are also observed on ER, MTs, and F-actin, although the margins are smaller for datasets with denser and more complex structures. SINO also outperforms FNO across all three metrics on all four datasets, demonstrating the effectiveness of the proposed architecture for the SIM reconstruction problems considered. Taken together, the image comparisons, residual maps, and quantitative metrics indicate that SINO reduces reconstruction errors while retaining fine spatial features across the four datasets. 

			\subsection{Application on Weather Forecasting}
			Weather forecasting relies on intricate mathematical and dynamic models that integrate multidisciplinary knowledge, including atmospheric physics, chemistry, and fluid dynamics. At the heart of the prediction system are numerical weather prediction (NWP) models, which are grounded in fundamental fluid dynamics and thermodynamics principles and leverage advanced computer simulation technologies to predict atmospheric conditions. Recently, artificial intelligence-driven weather forecasting methods have made significant strides, primarily utilizing large foundational models that exploit massive datasets and robust computing resources \cite{chen2023fuxi, bi2023accurate}. However, these complex systems often demand substantial infrastructure investments, limiting their broader application. Recent studies \cite{kurth2023fourcastnet, lin2023spherical} have introduced a variety of efficient and data-conserving approaches that require minimal input variables, sometimes relying on a single key atmospheric parameter, to achieve high-accuracy short-term predictions. Yet, enhancing the accuracy of medium-to-long-range forecasts remains one of the most cutting-edge and challenging issues in numerical weather prediction.
			
			\begin{figure}[htp!]
				\centering
				\includegraphics[width=1\linewidth]{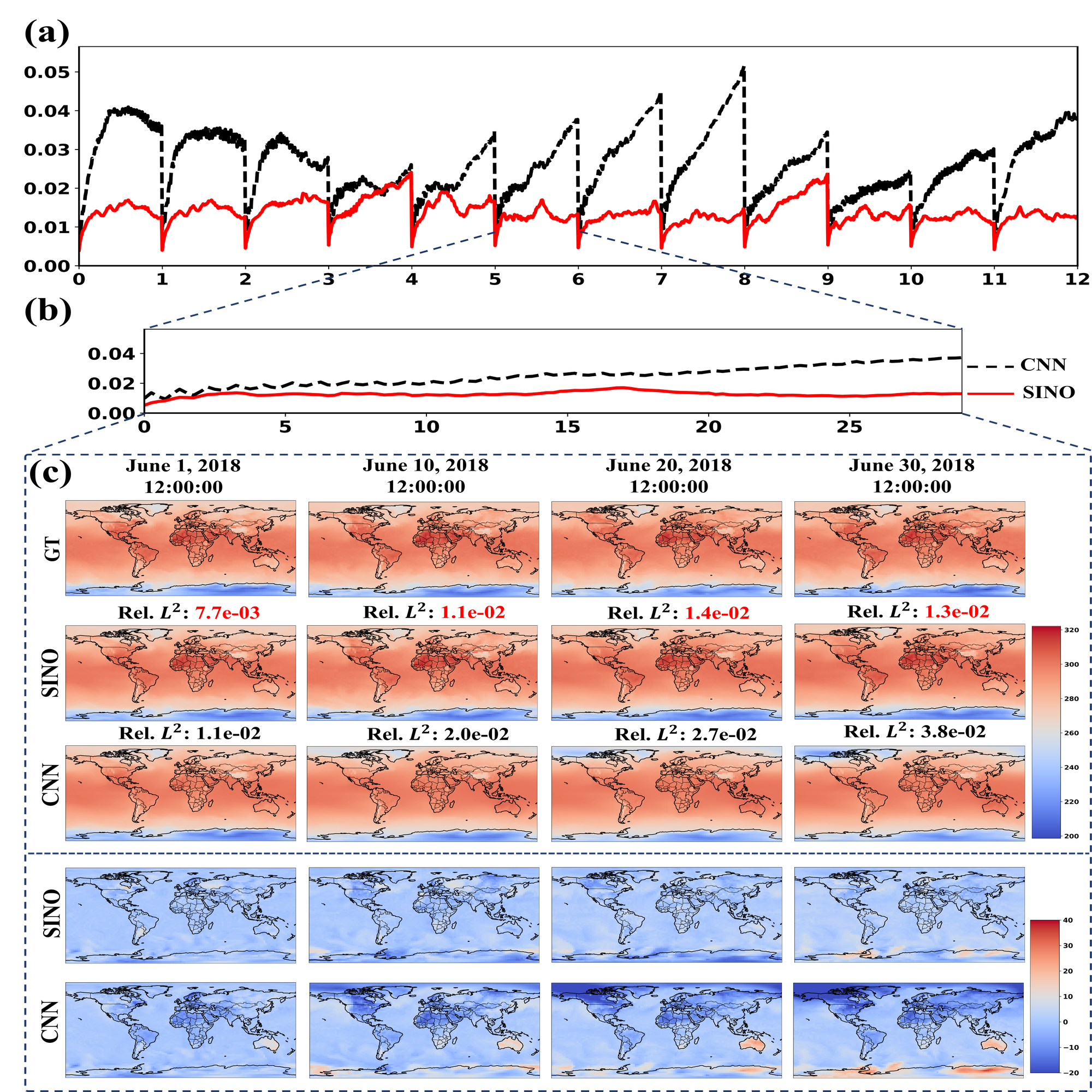}\\
				\vspace{0.3cm}
				\caption{\textbf{The comparison of Weather prediction results.}
					(a) Monthly prediction loss over the year 2018 comparing the proposed surrogate model (red) with the baseline model (black); (b) prediction loss curves for June; (c) visualization of selected time steps in June, where the first row shows the ground truth, the second row the predictions from SINO, the third row the predictions from the baseline model, and the fourth and fifth rows the residuals with respect to the ground truth for SINO and the baseline, respectively.}
				\label{fig:figure5}
			\end{figure}
			
			The numerical weather prediction (NWP) problem can be mathematically formulated as an initial value problem governed by a system of PDEs. We begin with the thermodynamic energy equation, which is derived from the first law of thermodynamics and forms an intrinsically coupled system with the momentum, continuity, and equation of state that describe atmospheric dynamic processes. Under the pressure coordinate system and assuming an ideal gas, the evolution of the temperature field is governed by the following equation:
			\begin{equation}
				c_p\left(\frac{\partial \boldsymbol{T}}{\partial t} + \mathbf{v} \cdot \nabla \boldsymbol{T} + \omega \frac{\partial \boldsymbol{T}}{\partial p}\right) - \frac{R \boldsymbol{T}}{p} \omega = Q,
				\label{eq:temperaturepde}
			\end{equation}
			where $\boldsymbol{T}$ denotes temperature (K), $\mathbf{v}=(u,v)$ is the horizontal wind velocity (m/s), $\omega = Dp/Dt$ is the vertical velocity in pressure coordinates (Pa/s), $c_p$ is the specific heat capacity at constant pressure (J/(kg·K)), $R$ is the gas constant for dry air (J/(kg·K)), $p$ is pressure (Pa), and $Q$ represents diabatic heating due to radiation, latent heat release, or turbulent fluxes (J/(kg·s)). This formulation highlights the interplay between temperature and other prognostic variables such as velocity, pressure, and diabatic processes. The temperature evolution Eq.~\eqref{eq:temperaturepde} can be discretized to yield a stepwise recursive relation for the temperature field:
			\begin{equation*}
				\mathcal{G}: \boldsymbol{T}(\mathbf{x},t) \mapsto \boldsymbol{T}(\mathbf{x},t+\Delta t),
			\end{equation*}
			which serves as the foundation for forward integration of prognostic temperature variables in numerical weather prediction models. To construct a data-driven forecasting model, we approximate the evolution operator $\mathcal{G}$ using a neural operator. 
			
			We use the publicly available \textbf{WeatherBench} \cite{rasp2020weatherbench}, which provides a comprehensive set of meteorological parameters across time. We focus on the near-surface air temperature at 2 meters (T2m) as our target variable, a widely recognized metric for assessing surface climate variability and thermally significant conditions relevant to human activities. Our analysis uses hourly T2m data with a temporal resolution of $\Delta t = 6$ hours. Data from the period 2010–2016 are used for training, observations from 2017 are used for validation, and data from 2018 are reserved for inference testing. The data are regridded to a spatial resolution of $128 \times 64$, corresponding to a uniform horizontal grid spacing of $2.8125^\circ$ in both longitude and latitude.

			We select the CNN model from WeatherBench \cite{rasp2020weatherbench} as the baseline method and conduct a numerical visualization comparison with our proposed method based on autoregressive inference over a one-month horizon, starting from the initial data of each month. Figure~\ref{fig:figure5}(a) presents a comparison between the predictions and ground truth obtained from both methods using the surrogate model for 12 consecutive months of autoregressive inference, along with the corresponding relative mean squared error curves. The analysis results indicate that the CNN method exhibits a significant error accumulation effect during long-term forecasting, with pronounced error growth between May and September. In contrast, our method does not show significant error accumulation throughout the entire forecasting period. Figure~\ref{fig:figure5}(b) further displays the error progression curve for June, demonstrating that our method not only achieves a significantly lower initial error than the CNN baseline but also maintains more stable error characteristics throughout the inference process. Figure~\ref{fig:figure5}(c) visualizes the predicted global temperature fields and corresponding residual maps at 12:00 UTC on June 1, 10, 20, and 30, with the relative L$^2$ error values annotated above each subplot. Quantitative analysis reveals that our method reduces the prediction error by approximately 30\% compared to the CNN method at the initial stage, and this performance advantage continues to expand over time, exceeding a 60\% error reduction by the end of June.

			\section{Conclusion and Future Works} \label{sec4}
			In this study, we introduced the Starter-Iterator Neural Operator (SINO),  which combines spectral initialization with learnable iterative refinement.  The Starter is designed to provide an initial approximation dominated by the principal spectral components, while the Iterator progressively refines the  solution through residual-correction updates. This decomposition is motivated  by classical iterative solvers and provides a structured connection between numerical iteration and neural operator learning. The theoretical analysis establishes compact-uniform approximation for the abstract Starter-Iterator
			decomposition in a fixed bounded linear setting.
			It also characterizes frequency-selective damping under
			common-eigenbasis and spectral-separation assumptions.
			The nonlinear PDE, underdetermined reconstruction, and
			partial-state forecasting experiments assess the practical
			architecture beyond this theoretical setting. We evaluated SINO on a range of forward and inverse problems, including time-dependent PDE prediction, image super-resolution, and weather forecasting. On the benchmarks considered in this work, SINO achieved lower approximation errors than the compared baselines in most experimental settings. The ablation and spectral studies indicate that the Starter and Iterator play complementary roles: spectral initialization reduces the initial approximation error, while iterative refinement improves the recovery of multiscale solution features. The autoregressive experiments also show that SINO can maintain stable predictive performance over the evaluated time horizons. 
			
			
			Future work will focus on three directions. First, we will investigate higher-order and structure-preserving numerical methods for generating more accurate training data. Second, we will explore parameterized iterators inspired by advanced numerical schemes, including Krylov-subspace methods, 
			to improve performance for ill-conditioned problems. Third, we will extend the architecture and its analysis to state-dependent nonlinear operators, regularized inverse problems, and coupled multiphysics systems. These extensions will require additional stability and approximation analysis for 
			nonlinear and iteration-dependent operator families.
			
			\section*{Acknowledgment}
			We would like to thank the anonymous reviewers for their constructive suggestions, which helped us improve the theoretical analysis. The work is supported by the National Natural Science Foundation of China (NSFC 12671658 and T2541053).

			\appendix
			
			\section{Theoretical evaluation}\label{secA1}
			\subsection{Proof of Theorem \ref{thm:fsno_approx}}\label{Them2}
			
			We first introduce the definitions and supporting lemmas used in the proof of Theorem~\ref{thm:fsno_approx}. The Sobolev spaces considered here are separable Hilbert spaces.
			\begin{definition}[Fourier projection and truncation operator]
				Let $s \ge 0$ and
				$\boldsymbol u\in H^s(\mathbb T^d;\mathbb R^{d_{0}})$ have
				the Fourier expansion
				\[
				\boldsymbol u(x)
				=
				\sum_{k\in\mathbb Z^d}
				\widehat{\boldsymbol u}(k)e^{\mathrm{i}\langle k,x\rangle}.
				\]
				For $K\in\mathbb N^+$, the \textbf{Fourier projection operator} is defined by
				\[
				\mathcal P_K\boldsymbol u
				:=
				\sum_{\lvert k\rvert_\infty<K}
				\widehat{\boldsymbol u}(k)e^{\mathrm{i}\langle k,x\rangle}.
				\]
				When necessary, we write
				\[
				\mathcal P_K^{\mathcal X}:\mathcal X\to\mathcal X,
				\qquad
				\mathcal P_K^{\mathcal Y}:\mathcal Y\to\mathcal Y
				\]
				to distinguish the projections acting on $\mathcal X$ and $\mathcal Y$, respectively.
				Both are finite-rank orthogonal projections and satisfy
				\[
				\|\mathcal P_K^{\mathcal X}\|_{\mathcal L(\mathcal X)}
				\leq 1,
				\qquad
				\|\mathcal P_K^{\mathcal Y}\|_{\mathcal L(\mathcal Y)}
				\leq 1.
				\]
				Moreover, they converge strongly to the identity as $K\to\infty$ and
				uniformly on compact subsets of the corresponding Sobolev spaces.
				
				
			\end{definition}
			
			\begin{definition}[Contraction operator]
				Let $\mathcal H$ be a Hilbert space. A bounded linear operator
				$\mathcal O\in\mathcal L(\mathcal H)$ is a strict contraction if there exists a constant $q\in(0,1)$ such that
				\[
				\|\mathcal O\boldsymbol u\|_{\mathcal H}
				\leq
				q\|\boldsymbol u\|_{\mathcal H},
				\qquad
				\forall\,\boldsymbol u\in\mathcal H.
				\]
				Equivalently,
				\[
				\|\mathcal O\|_{\mathcal L(\mathcal H)}
				\leq q<1.
				\]
				In an iterative scheme, $q$ determines the geometric decay rate of the propagated error.
			\end{definition}

			\begin{lemma}[Stability of the inverse and uniform convergence of the fixed-point iteration]
				\label{lem:freq_trunc_init_appro} 
				Let
				\[
				\mathcal A\in\mathcal L(\mathcal Y,\mathcal X),
				\qquad
				\mathcal B\in\mathcal L(\mathcal X,\mathcal Y),
				\]
				and assume that
				\[
				q
				:=
				\left\|
				\operatorname{Id}_{\mathcal Y}
				-
				\mathcal B\mathcal A
				\right\|_{\mathcal L(\mathcal Y)}
				<1.
				\]
				Then $\mathcal A$ is injective and bounded below on $\mathcal Y$,
				$\operatorname{Ran}(\mathcal A)$ is closed in $\mathcal X$, and
				\[
				\mathcal G: \operatorname{Ran}(\mathcal A)\to\mathcal Y,
				\qquad
				\mathcal A\mathcal G(f)=f,
				\]
				is a bounded linear operator satisfying
				\[
				\|\mathcal G\|_{\mathcal L(\operatorname{Ran}(\mathcal A),\mathcal Y)}
				\leq
				\frac{\|\mathcal B\|_{\mathcal L(\mathcal X,\mathcal Y)}}{1-q}.
				\]
				Let $\mathcal C\subset\operatorname{Ran}(\mathcal A)$ be compact
				with respect to the $\mathcal X$-norm, and let
				$\boldsymbol u^{(0)}:\mathcal C\to\mathcal Y$ be continuous.
				Define
				\begin{equation}
					\label{eq:iter_seq}
					\boldsymbol u^{(n+1)}(\boldsymbol f)
					=
					\mathcal I
					\left(
					\boldsymbol u^{(n)}(\boldsymbol f),
					\boldsymbol f
					\right)
					:=
					\boldsymbol u^{(n)}(\boldsymbol f)
					+
					\mathcal B
					\left(
					\boldsymbol f
					-
					\mathcal A\boldsymbol u^{(n)}(\boldsymbol f)
					\right), \qquad n \in \mathbb{N}
				\end{equation}
				Then for every $\epsilon>0$, there exists a single
				$N\in\mathbb N^+$ such that
				\[
				\sup_{\boldsymbol f\in\mathcal C}
				\left\|
				\mathcal G(\boldsymbol f)
				-
				\boldsymbol u^{(N)}(\boldsymbol f)
				\right\|_{\mathcal Y}
				<\epsilon.
				\]
			\end{lemma}
			\begin{proof}
				For any $\boldsymbol u\in\mathcal Y$, we have
				\[
				\begin{aligned}
					\|\boldsymbol u\|_{\mathcal Y}
					&\leq
					\left\|
					\left(
					\operatorname{Id}_{\mathcal Y}
					-
					\mathcal B\mathcal A
					\right)
					\boldsymbol u
					\right\|_{\mathcal Y}
					+
					\|\mathcal B\mathcal A\boldsymbol u\|_{\mathcal Y}\leq
					q\|\boldsymbol u\|_{\mathcal Y}
					+
					\|\mathcal B\|_{\mathcal L(\mathcal X,\mathcal Y)}
					\|\mathcal A\boldsymbol u\|_{\mathcal X}.
				\end{aligned}
				\]
				Hence
				\begin{equation}
					\label{eq:A_bounded_below}
					\|\mathcal A\boldsymbol u\|_{\mathcal X}
					\geq
					\frac{1-q}
					{\|\mathcal B\|_{\mathcal L(\mathcal X,\mathcal Y)}}
					\|\boldsymbol u\|_{\mathcal Y}.
				\end{equation}
				Thus $\mathcal A$ is bounded below and, in particular, injective.
				
				We next show that $\operatorname{Ran}(\mathcal A)$ is closed in $\mathcal X$.
				Let
				$\mathcal A\boldsymbol u_j\to\boldsymbol f$ in $\mathcal X$.
				By \eqref{eq:A_bounded_below},
				\[
				\|\boldsymbol u_j-\boldsymbol u_k\|_{\mathcal Y}
				\leq
				\frac{\|\mathcal B\|_{\mathcal L(\mathcal X,\mathcal Y)}}{1-q}
				\|
				\mathcal A\boldsymbol u_j
				-
				\mathcal A\boldsymbol u_k
				\|_{\mathcal X}.
				\]
				Therefore $\{\boldsymbol u_j\}$ is a Cauchy sequence in $\mathcal Y$.
				Since $\mathcal Y$ is complete, there exists
				$\boldsymbol u\in\mathcal Y$ such that
				$\boldsymbol u_j\to\boldsymbol u$ in $\mathcal Y$.
				The continuity of $\mathcal A$ then gives
				\[
				\mathcal A\boldsymbol u_j
				\to
				\mathcal A\boldsymbol u
				\qquad\text{in }\mathcal X.
				\]
				Hence $\boldsymbol f=\mathcal A\boldsymbol u\in\operatorname{Ran}(\mathcal A)$,
				and $\operatorname{Ran}(\mathcal A)$ is closed.
				
				Consequently,
				\[
				\mathcal G: \operatorname{Ran}(\mathcal A)\to\mathcal Y,
				\qquad
				\mathcal A\mathcal G(f)=f,
				\]
				is well defined and linear. Moreover, \eqref{eq:A_bounded_below} implies
				\[
				\|\mathcal G(\boldsymbol f)\|_{\mathcal Y}
				\leq
				\frac{\|\mathcal B\|_{\mathcal L(\mathcal X,\mathcal Y)}}{1-q}
				\|\boldsymbol f\|_{\mathcal X},
				\qquad
				\boldsymbol f\in\operatorname{Ran}(\mathcal A),
				\]
				so that $\mathcal G$ is bounded.
				
				For $\boldsymbol f\in\mathcal C$, define
				\[
				\boldsymbol e^{(n)}(\boldsymbol f)
				:=
				\mathcal G(\boldsymbol f)
				-
				\boldsymbol u^{(n)}(\boldsymbol f).
				\]
				Since
				$\mathcal A\mathcal G(\boldsymbol f)=\boldsymbol f$,
				Eq.~\eqref{eq:iter_seq} gives
				\[
				\begin{aligned}
					\boldsymbol e^{(n+1)}(\boldsymbol f)
					&=
					\left(
					\operatorname{Id}_{\mathcal Y}
					-
					\mathcal B\mathcal A
					\right)
					\boldsymbol e^{(n)}(\boldsymbol f), \qquad n=0, 1, \cdots
				\end{aligned}
				\]
				Therefore, we have
				\[
				\boldsymbol e^{(N)}(\boldsymbol f)
				=
				\left(
				\operatorname{Id}_{\mathcal Y}
				-
				\mathcal B\mathcal A
				\right)^N
				\boldsymbol e^{(0)}(\boldsymbol f),
				\]
				and hence
				\[
				\|\boldsymbol e^{(N)}(\boldsymbol f)\|_{\mathcal Y}
				\leq
				q^N
				\|\boldsymbol e^{(0)}(\boldsymbol f)\|_{\mathcal Y}.
				\]
				Since $\mathcal G|_{\mathcal C}$ and $\boldsymbol u^{(0)}$ are continuous
				and $\mathcal C$ is compact,
				\[
				M_0
				:=
				\sup_{\boldsymbol f\in\mathcal C}
				\left\|
				\mathcal G(\boldsymbol f)
				-
				\boldsymbol u^{(0)}(\boldsymbol f)
				\right\|_{\mathcal Y}
				<\infty.
				\]
				As $q \in (0, 1)$, we further have
				\[
				\sup_{\boldsymbol f\in\mathcal C}
				\|\boldsymbol e^{(N)}(\boldsymbol f)\|_{\mathcal Y}
				\leq
				q^N M_0
				\longrightarrow 0
				\qquad
				\text{as }N\to\infty,
				\]
				which proves the result.
			\end{proof}

			\begin{lemma}[Compact-uniform approximation of the iterator]
				\label{lem:univ_iter}
				Let $\mathcal C\subset\mathcal X$ be compact,
				let
				\[
				\mathcal A\in\mathcal L(\mathcal Y,\mathcal X),
				\qquad
				\mathcal B\in\mathcal L(\mathcal X,\mathcal Y),
				\]
				and let
				$\boldsymbol u^{(0)}:\mathcal C\to\mathcal Y$
				be continuous. 
				Then, for every $\epsilon>0$ and every fixed
				$N\in\mathbb N^+$, there exist finite-rank bounded linear operators
				\[
				\mathcal A_{\theta^{(n)}}
				\in\mathcal L(\mathcal Y,\mathcal X),
				\qquad
				\mathcal B_{\theta^{(n)}}
				\in\mathcal L(\mathcal X,\mathcal Y),
				\qquad
				n=0,\ldots,N-1,
				\]
				which define the step-dependent learned iteration
				\[
				\begin{aligned}
					\boldsymbol u_\theta^{(n+1)}(\boldsymbol f)
					=
					\mathcal I_{\theta^{(n)}}\left(
					\boldsymbol u_\theta^{(n)}(\boldsymbol f),
					\boldsymbol f
					\right)
					:=
					\boldsymbol u_\theta^{(n)}(\boldsymbol f)
					+
					\mathcal B_{\theta^{(n)}}
					\left(
					\boldsymbol f
					-
					\mathcal A_{\theta^{(n)}}
					\boldsymbol u_\theta^{(n)}(\boldsymbol f)
					\right),
					\qquad
					n=0,\ldots,N-1,
				\end{aligned}
				\]
				with $\boldsymbol u_\theta^{(0)}(\boldsymbol f)
				= \boldsymbol u^{(0)}(\boldsymbol f)$, 
				such that
				\[
				\sup_{\boldsymbol f\in\mathcal C}
				\left\|
				\mathcal I^N\left(
				\boldsymbol u^{(0)}(\boldsymbol f);
				\boldsymbol f
				\right)
				-
				\mathcal I_\theta^N\left(
				\boldsymbol u^{(0)}(\boldsymbol f);
				\boldsymbol f
				\right)
				\right\|_{\mathcal Y}
				<\epsilon.
				\]
				Here, $\mathcal I^N$ denotes $N$ successive applications of the exact
				iteration rule and $\mathcal I_\theta^N$ denotes the ordered application of $N$
				iteration steps with iteration-dependent parameters, defined as Eq.~\eqref{eq:multi_step_learnable_iterator}.
			\end{lemma}
			
			\begin{proof}
				For each $\boldsymbol f\in\mathcal C$, let
				$\{\boldsymbol u^{(n)}(\boldsymbol f)\}_{n=0}^{N}$ denote the exact
				iteration generated by
				\[
				\boldsymbol u^{(n+1)}(\boldsymbol f)
				=
				\boldsymbol u^{(n)}(\boldsymbol f)
				+
				\mathcal B
				\left(
				\boldsymbol f
				-
				\mathcal A\boldsymbol u^{(n)}(\boldsymbol f)
				\right),
				\qquad
				n=0,\ldots,N-1,
				\]
				with the same initial state
				$\boldsymbol u^{(0)}(\boldsymbol f)$ as the learned iterator.
				
				Since $\boldsymbol u^{(0)}$ is continuous and $\mathcal A$ and
				$\mathcal B$ are bounded linear operators, the map $\boldsymbol f
				\longmapsto\boldsymbol u^{(n)}(\boldsymbol f)$ is continuous for every fixed $n$. Because $\mathcal C$ is compact and
				$N$ is finite, the trajectory and residual sets
				\[
				\mathcal U_N
				:=
				\bigcup_{n=0}^{N-1}
				\left\{
				\boldsymbol u^{(n)}(\boldsymbol f):
				\boldsymbol f\in\mathcal C
				\right\}
				\subset\mathcal Y, \qquad
				\mathcal R_N
				:=
				\bigcup_{n=0}^{N-1}
				\left\{
				\boldsymbol f
				-
				\mathcal A\boldsymbol u^{(n)}(\boldsymbol f):
				\boldsymbol f\in\mathcal C
				\right\}
				\subset\mathcal X
				\]
				are compact. 
				
				{ Since $\mathcal X$ and $\mathcal Y$ are separable Hilbert spaces,
					finite-rank orthogonal projections converge strongly to the identity and
					uniformly on compact subsets. By composing $\mathcal A$ and $\mathcal B$
					with such projections, for every $\delta>0$ and
					$n=0,\ldots,N-1$, we may choose finite-rank bounded linear operators
					$\mathcal A_{\theta^{(n)}}$ and
					$\mathcal B_{\theta^{(n)}}$ such that
					\[
					\sup_{\boldsymbol u\in\mathcal U_N}
					\left\|
					\left(
					\mathcal A-\mathcal A_{\theta^{(n)}}
					\right)\boldsymbol u
					\right\|_{\mathcal X}
					\leq\delta, \qquad
					\sup_{\boldsymbol r\in\mathcal R_N}
					\left\|
					\left(
					\mathcal B-\mathcal B_{\theta^{(n)}}
					\right)\boldsymbol r
					\right\|_{\mathcal Y}
					\leq\delta.
					\]
					These are compact-uniform approximation bounds rather than global
					operator-norm approximation bounds. }
				Because the orthogonal projections used in this construction have norm at
				most one, the approximating operators may also be chosen such that
				\[
				\left\|
				\mathcal A_{\theta^{(n)}}
				\right\|_{\mathcal L(\mathcal Y,\mathcal X)}
				\leq \left\|
				\mathcal A
				\right\|_{\mathcal L(\mathcal Y,\mathcal X)},
				\qquad
				\left\|
				\mathcal B_{\theta^{(n)}}
				\right\|_{\mathcal L(\mathcal X,\mathcal Y)}
				\leq \left\|
				\mathcal B
				\right\|_{\mathcal L(\mathcal X,\mathcal Y)},
				\]
				Set $\Lambda
				:=
				\max
				\left\{
				\|\mathcal A\|_{\mathcal L(\mathcal Y,\mathcal X)},
				\|\mathcal B\|_{\mathcal L(\mathcal X,\mathcal Y)}
				\right\}$, 
				then
				\[
				\left\|
				\operatorname{Id}_{\mathcal Y}
				-
				\mathcal B_{\theta^{(n)}}
				\mathcal A_{\theta^{(n)}}
				\right\|_{\mathcal L(\mathcal Y)}
				\leq
				1+\Lambda^2
				=:M.
				\]
				Here the constants $\Lambda$ and $M$ are both independent of
				both $n$ and $\delta$.
				For $\boldsymbol f\in\mathcal C$, we define
				\[
				\boldsymbol e^{(n)}(\boldsymbol f)
				:=
				\boldsymbol u^{(n)}(\boldsymbol f)
				-
				\boldsymbol u_\theta^{(n)}(\boldsymbol f).
				\]
				Since the exact and learned iterations start from the same initial state, i.e. $\boldsymbol e^{(0)}(\boldsymbol f)=\boldsymbol 0$.
				Subtracting the two iteration rules gives
				\[
				\begin{aligned}
					\boldsymbol e^{(n+1)}(\boldsymbol f)
					={}&
					\left(
					\operatorname{Id}_{\mathcal Y}
					-
					\mathcal B_{\theta^{(n)}}
					\mathcal A_{\theta^{(n)}}
					\right)
					\boldsymbol e^{(n)}(\boldsymbol f)+
					\left(
					\mathcal B-\mathcal B_{\theta^{(n)}}
					\right)
					\big(
					\boldsymbol f
					-
					\mathcal A\boldsymbol u^{(n)}(\boldsymbol f)
					\big)+
					\mathcal B_{\theta^{(n)}}
					\left(
					\mathcal A_{\theta^{(n)}}-\mathcal A
					\right)
					\boldsymbol u^{(n)}(\boldsymbol f).
				\end{aligned}
				\]
				Taking the norm in $\mathcal Y$ yields
				\[
				\begin{aligned}
					\big\|
					\boldsymbol e^{(n+1)}(\boldsymbol f)
					\big\|_{\mathcal Y}
					\leq &
					\big\|
					\operatorname{Id}_{\mathcal Y}
					-
					\mathcal B_{\theta^{(n)}}
					\mathcal A_{\theta^{(n)}}
					\big\|_{\mathcal L(\mathcal Y)}
					\big\|
					\boldsymbol e^{(n)}(\boldsymbol f)
					\big\|_{\mathcal Y}
					+
					\big\|
					\left(
					\mathcal B-\mathcal B_{\theta^{(n)}}
					\right)
					\big(
					\boldsymbol f
					-
					\mathcal A\boldsymbol u^{(n)}(\boldsymbol f)
					\big)
					\big\|_{\mathcal Y}
					\\
					&+
					\big\|
					\mathcal B_{\theta^{(n)}}
					\big\|_{\mathcal L(\mathcal X,\mathcal Y)}
					\big\|
					\left(
					\mathcal A_{\theta^{(n)}}-\mathcal A
					\right)
					\boldsymbol u^{(n)}(\boldsymbol f)
					\big\|_{\mathcal X}.
				\end{aligned}
				\]
				Using the compact-uniform approximation bounds, we obtain
				\[
				\big\|
				\boldsymbol e^{(n+1)}(\boldsymbol f)
				\big\|_{\mathcal Y}
				\leq
				M
				\big\|
				\boldsymbol e^{(n)}(\boldsymbol f)
				\big\|_{\mathcal Y}
				+
				(1+\Lambda)\delta.
				\]
				Since $\boldsymbol e^{(0)}(\boldsymbol f)=\boldsymbol 0$, recursive application gives
				\[
				\big\|
				\boldsymbol e^{(N)}(\boldsymbol f)
				\big\|_{\mathcal Y}
				\leq
				(1+\Lambda)\delta
				\sum_{k=0}^{N-1}M^k.
				\]
				The bound is uniform over $\boldsymbol f\in\mathcal C$. Since $N$ is fixed, we may choose $\delta>0$ sufficiently small such that
				\[
				(1+\Lambda)\delta
				\sum_{k=0}^{N-1}M^k
				<\epsilon.
				\]
				Consequently,
				\[
				\sup_{\boldsymbol f\in\mathcal C}
				\Big\|
				\mathcal I^N\left(
				\boldsymbol u^{(0)}(\boldsymbol f);
				\boldsymbol f
				\right)
				-
				\mathcal I_\theta^N\big(
				\boldsymbol u^{(0)}(\boldsymbol f);
				\boldsymbol f
				\big)
				\Big\|_{\mathcal Y}
				<\epsilon,
				\]
				which completes the proof.
			\end{proof}
			
			\begin{lemma}[Universal approximation property of the starter operator]
				\label{lem:fno_universal}
				Let $\mathcal C\subset\operatorname{Ran}(\mathcal A)$
				be compact with respect to the $\mathcal X$-norm, and let
				\[
				\mathcal G: \operatorname{Ran}(\mathcal A)\to\mathcal Y,
				\qquad
				\mathcal A\mathcal G(f)=f,
				\]
				be the bounded solution operator obtained in
				Lemma~\ref{lem:freq_trunc_init_appro}.
				Since $\operatorname{Ran}(\mathcal A)$ is closed in the Hilbert space
				$\mathcal X$, let $\Pi_{\mathcal A}: \mathcal X\to\operatorname{Ran}(\mathcal A)$ denote the orthogonal projection. For $K\in\mathbb N^+$, define the Fourier-truncated operator
				$\mathcal G_K:\mathcal X\to\mathcal Y$ by
				\begin{equation}
					\label{eq:operator_proj}
					\mathcal G_K(\boldsymbol f)
					:=
					\mathcal P_K^{\mathcal Y}
					\mathcal G
					\left(
					\Pi_{\mathcal A}
					\mathcal P_K^{\mathcal X}
					\boldsymbol f
					\right).
				\end{equation}
				Then $\mathcal G_K$ is a bounded linear operator from
				$\mathcal X$ to $\mathcal Y$. Moreover,
				\[
				\lim_{K\to\infty}
				\sup_{\boldsymbol f\in\mathcal C}
				\|
				\mathcal G(\boldsymbol f)
				-
				\mathcal G_K(\boldsymbol f)
				\|_{\mathcal Y}
				=0.
				\]
				Under the standard assumptions of the FNO universal approximation
				theorem, for every $\epsilon>0$ and every fixed $K\in\mathbb N^+$,
				there exists an FNO starter
				$\mathcal S_\theta:\mathcal X\to\mathcal Y$ such that
				\[
				\sup_{\boldsymbol f\in\mathcal C}
				\left\|
				\mathcal G_K(\boldsymbol f)
				-
				\mathcal S_\theta(\boldsymbol f)
				\right\|_{\mathcal Y}
				<\epsilon.
				\]
			\end{lemma}
			
			\begin{proof}
				Since $\operatorname{Ran}(\mathcal A)$ is a closed subspace of the
				Hilbert space $\mathcal X$, the orthogonal projection
				$\Pi_{\mathcal A}$ is well defined and satisfies
				\[
				\|\Pi_{\mathcal A}\|_{\mathcal L(\mathcal X)}\leq 1.
				\]
				Hence Eq.~\eqref{eq:operator_proj} defines a bounded linear operator
				$\mathcal G_K:\mathcal X\to\mathcal Y$.
				
				For $\boldsymbol f\in\mathcal C\subset\operatorname{Ran}(\mathcal A)$,
				we have
				$\Pi_{\mathcal A}\boldsymbol f=\boldsymbol f$. Therefore,
				\[
				\|
				\mathcal G(\boldsymbol f)
				-
				\mathcal G_K(\boldsymbol f)
				\|_{\mathcal Y}
				\leq
				\|
				(\operatorname{Id}_{\mathcal Y}
				-
				\mathcal P_K^{\mathcal Y})
				\mathcal G(\boldsymbol f)
				\|_{\mathcal Y}
				+
				\|\mathcal G\|_{\mathcal L(\operatorname{Ran}(\mathcal A),\mathcal Y)}
				\|
				\boldsymbol f
				-
				\Pi_{\mathcal A}
				\mathcal P_K^{\mathcal X}
				\boldsymbol f
				\|_{\mathcal X}.
				\]
				Since $\Pi_{\mathcal A}\boldsymbol f=\boldsymbol f$ and
				$\|\Pi_{\mathcal A}\|\leq1$,
				\[
				\|
				\boldsymbol f
				-
				\Pi_{\mathcal A}
				\mathcal P_K^{\mathcal X}\boldsymbol f
				\|_{\mathcal X}
				\leq
				\|
				(\operatorname{Id}_{\mathcal X}
				-
				\mathcal P_K^{\mathcal X})
				\boldsymbol f
				\|_{\mathcal X}.
				\]
				Because $\mathcal C$ is compact in $\mathcal X$ and
				$\mathcal G(\mathcal C)$ is compact in $\mathcal Y$, the Fourier
				projections converge uniformly to the identity on these two compact
				sets. Consequently,
				\[
				\sup_{\boldsymbol f\in\mathcal C}
				\|
				\mathcal G(\boldsymbol f)
				-
				\mathcal G_K(\boldsymbol f)
				\|_{\mathcal Y}
				\longrightarrow0
				\qquad
				\text{as }K\to\infty.
				\]
				Finally, $\mathcal G_K:\mathcal X\to\mathcal Y$ is continuous.
				The universal approximation theorem for Fourier neural operators
				\cite[Theorem~5]{kovachki2021universal}
				therefore yields, under its stated assumptions on the activation
				function and network architecture, an FNO
				$\mathcal S_\theta:\mathcal X\to\mathcal Y$ satisfying
				\[
				\sup_{\boldsymbol f\in\mathcal C}
				\left\|
				\mathcal G_K(\boldsymbol f)
				-
				\mathcal S_\theta(\boldsymbol f)
				\right\|_{\mathcal Y}
				<\epsilon.
				\]
			\end{proof}
			
			Building upon the preceding results, we now prove the main approximation theorem for SINO. The proof combines the uniform convergence of the exact fixed-point iteration from
			Lemma~\ref{lem:freq_trunc_init_appro}, the compact-uniform approximation of the iterator from Lemma~\ref{lem:univ_iter}, and the FNO approximation of the Fourier-truncated initialization from
			Lemma~\ref{lem:fno_universal}.
			
			\begin{proof}[Proof of Theorem~\ref{thm:fsno_approx}]
				By Lemma~\ref{lem:freq_trunc_init_appro},
				$\mathcal A$ is injective and bounded below,
				$\operatorname{Ran}(\mathcal A)$ is closed in $\mathcal X$, and
				\[
				\mathcal G:=\mathcal A^{-1}
				\in
				\mathcal L
				\left(
				\operatorname{Ran}(\mathcal A),
				\mathcal Y
				\right).
				\]
				Fix $K\in\mathbb N^+$ and let
				$\mathcal G_K:\mathcal X\to\mathcal Y$
				be the Fourier-truncated operator defined in
				Lemma~\ref{lem:fno_universal}.
				Since $\mathcal G_K|_{\mathcal C}$ is continuous, it may be used as
				the initial-state map in
				Lemma~\ref{lem:freq_trunc_init_appro}. Define
				\[
				M_{0,K}
				:=
				\sup_{\boldsymbol f\in\mathcal C}
				\left\|
				\mathcal G(\boldsymbol f)
				-
				\mathcal G_K(\boldsymbol f)
				\right\|_{\mathcal Y}.
				\]
				Since $\mathcal C$ is compact and both
				$\mathcal G|_{\mathcal C}$ and $\mathcal G_K|_{\mathcal C}$
				are continuous, we have $M_{0,K}<\infty$.
				Then we can choose $N\in\mathbb N^+$ sufficiently large such that
				\[
				q^N M_{0,K}
				<
				\frac{\epsilon}{3}.
				\]
				Lemma~\ref{lem:freq_trunc_init_appro} then gives
				\begin{equation}
					\label{eq:proof_error_I}
					\sup_{\boldsymbol f\in\mathcal C}
					\left\|
					\mathcal G(\boldsymbol f)
					-
					\mathcal I^N
					\left(
					\mathcal G_K(\boldsymbol f);
					\boldsymbol f
					\right)
					\right\|_{\mathcal Y}
					<
					\frac{\epsilon}{3}.
				\end{equation}
				For this fixed $N$, Lemma~\ref{lem:univ_iter}, applied with the
				initial-state map
				$\boldsymbol u^{(0)}=\mathcal G_K|_{\mathcal C}$,
				provides finite-rank learned iteration operators such that
				\begin{equation}
					\label{eq:proof_error_II}
					\sup_{\boldsymbol f\in\mathcal C}
					\left\|
					\mathcal I^N
					\left(
					\mathcal G_K(\boldsymbol f);
					\boldsymbol f
					\right)
					-
					\mathcal I_\theta^N
					\left(
					\mathcal G_K(\boldsymbol f);
					\boldsymbol f
					\right)
					\right\|_{\mathcal Y}
					<
					\frac{\epsilon}{3}.
				\end{equation}
				For each learned iteration step, we define
				\[
				\mathcal E_{\theta^{(n)}}
				:=
				\operatorname{Id}_{\mathcal Y}
				-
				\mathcal B_{\theta^{(n)}}
				\mathcal A_{\theta^{(n)}}, \qquad n \in \mathbb N.
				\]
				So for any $\boldsymbol u_1,\boldsymbol u_2\in\mathcal Y$ and
				$\boldsymbol f\in\mathcal C$, 
				\[ 
				\mathcal I_{\theta^{(n)}}
				(\boldsymbol u_1,\boldsymbol f)
				-
				\mathcal I_{\theta^{(n)}}
				(\boldsymbol u_2,\boldsymbol f)
				=
				\mathcal E_{\theta^{(n)}}
				(\boldsymbol u_1-\boldsymbol u_2).
				\]
				Consequently,
				\[
				\left\|
				\mathcal I_\theta^N(\boldsymbol u_1;\boldsymbol f)
				-
				\mathcal I_\theta^N(\boldsymbol u_2;\boldsymbol f)
				\right\|_{\mathcal Y}
				\leq
				L_N
				\|\boldsymbol u_1-\boldsymbol u_2\|_{\mathcal Y},
				\]
				where
				\[
				L_N
				:=
				\prod_{n=0}^{N-1}
				\left\|
				\mathcal E_{\theta^{(n)}}
				\right\|_{\mathcal L(\mathcal Y)}
				<\infty.
				\]
				By Lemma~\ref{lem:fno_universal}, the starter
				$\mathcal S_\theta:\mathcal X\to\mathcal Y$
				may be chosen such that
				\[
				\sup_{\boldsymbol f\in\mathcal C}
				\left\|
				\mathcal G_K(\boldsymbol f)
				-
				\mathcal S_\theta(\boldsymbol f)
				\right\|_{\mathcal Y}
				<
				\frac{\epsilon}{3\max\{1,L_N\}}.
				\]
				It follows that
				\begin{equation}
					\label{eq:proof_error_III}
					\sup_{\boldsymbol f\in\mathcal C}
					\left\|
					\mathcal I_\theta^N
					\left(
					\mathcal G_K(\boldsymbol f);
					\boldsymbol f
					\right)
					-
					\mathcal I_\theta^N
					\left(
					\mathcal S_\theta(\boldsymbol f);
					\boldsymbol f
					\right)
					\right\|_{\mathcal Y}
					<
					\frac{\epsilon}{3}.
				\end{equation}
				As $\mathcal G_\theta$ is defined as Eq.~\eqref{sino_form}, for any $\boldsymbol f\in\mathcal C$, the triangle inequality gives
				\[
				\begin{aligned}
					\left\|
					\mathcal G(\boldsymbol f)
					-
					\mathcal G_\theta(\boldsymbol f)
					\right\|_{\mathcal Y}
					\leq{}&
					\left\|
					\mathcal G(\boldsymbol f)
					-
					\mathcal I^N
					\left(
					\mathcal G_K(\boldsymbol f);
					\boldsymbol f
					\right)
					\right\|_{\mathcal Y}
					\\
					&+
					\left\|
					\mathcal I^N
					\left(
					\mathcal G_K(\boldsymbol f);
					\boldsymbol f
					\right)
					-
					\mathcal I_\theta^N
					\left(
					\mathcal G_K(\boldsymbol f);
					\boldsymbol f
					\right)
					\right\|_{\mathcal Y} \\
					&+
					\left\|
					\mathcal I_\theta^N
					\left(
					\mathcal G_K(\boldsymbol f);
					\boldsymbol f
					\right)
					-
					\mathcal I_\theta^N
					\left(
					\mathcal S_\theta(\boldsymbol f);
					\boldsymbol f
					\right)
					\right\|_{\mathcal Y} \leq \epsilon.
				\end{aligned}
				\]
				Thus we finally get
				\begin{equation*}
					\sup_{\boldsymbol f\in\mathcal C}
					\left\|
					\mathcal G(\boldsymbol f)
					-
					\mathcal G_\theta(\boldsymbol f)
					\right\|_{\mathcal Y}
					<\epsilon.
				\end{equation*}
				This completes the proof.
			\end{proof}

			\begin{remark}[Difference between our proof and ~\cite{kovachki2021universal}]
				The FNO-based Starter provides an approximation of the
				Fourier-truncated operator $\mathcal G_K$ \cite{kovachki2021universal}, while the Iterator progressively refines this initialization toward the target operator $\mathcal G$. Thus, SINO combines spectral initialization with iterative correction rather than relying on the Starter alone. The numerical results in Section~\ref{sec:model_eva} demonstrate the effectiveness of this decomposition on the considered benchmarks.
			\end{remark}

			\subsection{Frequency-selective damping in stationary iterations}
			\label{sec:Spectral_analysis}
			
			\begingroup
	
			\begin{theorem}[Frequency-selective spectral damping]
				\label{thm:spectral_convergence_operator}
				
				Let $\mathcal X$ and $\mathcal Y$ be Hilbert spaces, and let $\mathcal A\in\mathcal L(\mathcal Y,\mathcal X)$, $\mathcal B\in\mathcal L(\mathcal X,\mathcal Y)$. 
				Consider the stationary residual-correction iteration
				\[
				\boldsymbol u^{(n+1)}
				=
				\boldsymbol u^{(n)}
				+
				\mathcal B
				\left(
				\boldsymbol f-\mathcal A\boldsymbol u^{(n)}
				\right),
				\qquad n=0,1,\ldots,
				\]
				and define the error-propagation operator
				\[
				\mathcal E
				:=
				\mathrm{Id}_{\mathcal Y}-\mathcal B\mathcal A
				\in\mathcal L(\mathcal Y).
				\]
				Assume that:
				
				\begin{enumerate}
					\item[(i)]
					There exists $\boldsymbol u^*\in\mathcal Y$ satisfying
					\[
					\mathcal A\boldsymbol u^*=\boldsymbol f.
					\]
					
					\item[(ii)]
					There exists a densely defined, nonnegative, self-adjoint spatial
					operator $\mathcal L$ on $\mathcal Y$ with a complete orthonormal
					eigenbasis
					$\{\boldsymbol\phi_j\}_{j=1}^{\infty}$ such that
					\[
					\mathcal L\boldsymbol\phi_j
					=
					\lambda_j\boldsymbol\phi_j,
					\qquad
					0\leq\lambda_1\leq\lambda_2\leq\cdots,
					\qquad
					\lambda_j\to\infty.
					\]
					The ordering of $\lambda_j$ is used to characterize the spatial
					frequency of $\boldsymbol\phi_j$, with larger eigenvalues corresponding
					to higher-frequency modes.
					
					\item[(iii)]
					The operator $\mathcal E$ is a strict contraction and is diagonal with
					respect to the same orthonormal basis:
					\[
					\mathcal E\boldsymbol\phi_j
					=
					\eta_j\boldsymbol\phi_j,
					\qquad
					\|\mathcal E\|_{\mathcal L(\mathcal Y)}<1.
					\]
					Here, $\eta_j$ is the modal amplification factor associated with
					$\boldsymbol\phi_j$.
					
					\item[(iv)]
					Let $\lambda_{\mathrm c}\geq0$ be a prescribed cutoff and define the
					nonempty low- and high-frequency index sets
					\[
					J_{\mathrm L}
					:=
					\{j:\lambda_j\leq\lambda_{\mathrm c}\},
					\qquad
					J_{\mathrm H}
					:=
					\{j:\lambda_j>\lambda_{\mathrm c}\}.
					\]
					Define the corresponding orthogonal projections by
					\[
					P_{\mathrm L}\boldsymbol u
					:=
					\sum_{j\in J_{\mathrm L}}
					\langle\boldsymbol u,\boldsymbol\phi_j\rangle_{\mathcal Y}
					\boldsymbol\phi_j,
					\qquad
					P_{\mathrm H}\boldsymbol u
					:=
					\sum_{j\in J_{\mathrm H}}
					\langle\boldsymbol u,\boldsymbol\phi_j\rangle_{\mathcal Y}
					\boldsymbol\phi_j.
					\]
					
					\item[(v)]
					Define $q_{\mathrm H}
					:=
					\sup_{j\in J_{\mathrm H}}|\eta_j|$ and $q_{\mathrm L}
					:=
					\inf_{j\in J_{\mathrm L}}|\eta_j|$. Assume the spectral-separation condition
					\[
					0\leq q_{\mathrm H}<q_{\mathrm L}<1.
					\]
				\end{enumerate}
				Then the following statements hold:
				
				\begin{enumerate}
					\item[(a)]
					Every modal amplification factor satisfies
					\[
					|\eta_j|
					\leq
					\|\mathcal E\|_{\mathcal L(\mathcal Y)}
					<1,
					\qquad j=1,2,\ldots.
					\]
					
					\item[(b)]
					Let $\boldsymbol e^{(n)}:=\boldsymbol u^* - \boldsymbol u^{(n)}$. 
					If $P_{\mathrm L}\boldsymbol e^{(0)}\neq0$, then, for every $n\geq0$,
					\[
					\frac{
						\|P_{\mathrm H}\boldsymbol e^{(n)}\|_{\mathcal Y}
					}{
						\|P_{\mathrm L}\boldsymbol e^{(n)}\|_{\mathcal Y}
					}
					\leq
					\left(
					\frac{q_{\mathrm H}}{q_{\mathrm L}}
					\right)^n
					\frac{
						\|P_{\mathrm H}\boldsymbol e^{(0)}\|_{\mathcal Y}
					}{
						\|P_{\mathrm L}\boldsymbol e^{(0)}\|_{\mathcal Y}
					}.
					\]
					Consequently,
					\[
					\lim_{n\to\infty}
					\frac{
						\|P_{\mathrm H}\boldsymbol e^{(n)}\|_{\mathcal Y}
					}{
						\|P_{\mathrm L}\boldsymbol e^{(n)}\|_{\mathcal Y}
					}
					=0.
					\]
				\end{enumerate}
				Thus, under the common-eigenbasis and spectral-separation assumptions, the high-frequency component of the error is damped asymptotically faster than the low-frequency component.
			\end{theorem}
			\endgroup
			
			\begin{proof}
				With $e^{(n)}:=u^\ast-u^{(n)}$ and $\mathcal A u^\ast=f$,
				we have
				\[
				\begin{aligned}
					e^{(n+1)}
					&=
					u^\ast-u^{(n)}
					-\mathcal B\bigl(f-\mathcal A u^{(n)}\bigr)\\
					&=
					e^{(n)}-\mathcal B\mathcal A e^{(n)}\\
					&=
					\bigl(\operatorname{Id}_{\mathcal Y}
					-\mathcal B\mathcal A\bigr)e^{(n)}
					=
					\mathcal E e^{(n)}.
				\end{aligned}
				\]
				Hence $e^{(n)}=\mathcal E^n e^{(0)}$.
				Since $\{\boldsymbol\phi_j\}_{j=1}^{\infty}$ is an orthonormal basis of $\mathcal Y$, $\|\boldsymbol\phi_j\|_{\mathcal Y}=1$. Using $\mathcal E\boldsymbol\phi_j
				=
				\eta_j\boldsymbol\phi_j,$
				we obtain
				\[
				|\eta_j|
				=
				\|\eta_j\boldsymbol\phi_j\|_{\mathcal Y}
				=
				\|\mathcal E\boldsymbol\phi_j\|_{\mathcal Y}
				\leq
				\|\mathcal E\|_{\mathcal L(\mathcal Y)}
				\|\boldsymbol\phi_j\|_{\mathcal Y}.
				\]
				Hence,
				\[
				|\eta_j|
				\leq
				\|\mathcal E\|_{\mathcal L(\mathcal Y)}
				<1,
				\qquad
				j=1,2,\ldots,
				\]
				which proves statement~(a).
				
				To prove statement~(b), expand the initial error in the common
				orthonormal basis:
				\begin{equation*}
					\boldsymbol e^{(0)}
					=
					\sum_{j=1}^{\infty}
					c_j\boldsymbol\phi_j,
					\qquad
					c_j
					=
					\left\langle
					\boldsymbol e^{(0)},\boldsymbol\phi_j
					\right\rangle_{\mathcal Y}.
					\label{eq:error_initial_expansion}
				\end{equation*}
				Since
				$\mathcal E\boldsymbol\phi_j=\eta_j\boldsymbol\phi_j$,
				it follows that
				\begin{equation*}
					\boldsymbol e^{(n)}
					=
					\sum_{j=1}^{\infty}
					\eta_j^n c_j\boldsymbol\phi_j.
					\label{eq:error_modal_expansion}
				\end{equation*}
				For the high-frequency component,
				\[
				P_{\mathrm H}\boldsymbol e^{(n)}
				=
				\sum_{j\in J_{\mathrm H}}
				\eta_j^n c_j\boldsymbol\phi_j.
				\]
				By orthonormality and
				$q_{\mathrm H}=\sup_{j\in J_{\mathrm H}}|\eta_j|$, we have
				\[
				\begin{aligned}
					\left\|
					P_{\mathrm H}\boldsymbol e^{(n)}
					\right\|_{\mathcal Y}^2
					=
					\sum_{j\in J_{\mathrm H}}
					|\eta_j|^{2n}|c_j|^2\leq
					q_{\mathrm H}^{2n}
					\sum_{j\in J_{\mathrm H}}|c_j|^2
					=
					q_{\mathrm H}^{2n}
					\left\|
					P_{\mathrm H}\boldsymbol e^{(0)}
					\right\|_{\mathcal Y}^2.
				\end{aligned}
				\]
				Therefore,
				\begin{equation}
					\left\|
					P_{\mathrm H}\boldsymbol e^{(n)}
					\right\|_{\mathcal Y}
					\leq
					q_{\mathrm H}^n
					\left\|
					P_{\mathrm H}\boldsymbol e^{(0)}
					\right\|_{\mathcal Y}.
					\label{eq:high_frequency_decay}
				\end{equation}
				Similarly,
				\[
				P_{\mathrm L}\boldsymbol e^{(n)}
				=
				\sum_{j\in J_{\mathrm L}}
				\eta_j^n c_j\boldsymbol\phi_j.
				\]
				Using
				$q_{\mathrm L}=\inf_{j\in J_{\mathrm L}}|\eta_j|$, we obtain
				\[
				\begin{aligned}
					\left\|
					P_{\mathrm L}\boldsymbol e^{(n)}
					\right\|_{\mathcal Y}^2
					=
					\sum_{j\in J_{\mathrm L}}
					|\eta_j|^{2n}|c_j|^2
					\geq
					q_{\mathrm L}^{2n}
					\sum_{j\in J_{\mathrm L}}|c_j|^2
					=
					q_{\mathrm L}^{2n}
					\left\|
					P_{\mathrm L}\boldsymbol e^{(0)}
					\right\|_{\mathcal Y}^2.
				\end{aligned}
				\]
				Hence,
				\begin{equation}
					\left\|
					P_{\mathrm L}\boldsymbol e^{(n)}
					\right\|_{\mathcal Y}
					\geq
					q_{\mathrm L}^n
					\left\|
					P_{\mathrm L}\boldsymbol e^{(0)}
					\right\|_{\mathcal Y}.
					\label{eq:low_frequency_decay}
				\end{equation}
				If
				$P_{\mathrm L}\boldsymbol e^{(0)}\neq0$, then
				$q_{\mathrm L}>0$ and
				\[
				\left\|
				P_{\mathrm L}\boldsymbol e^{(n)}
				\right\|_{\mathcal Y}
				>0,
				\qquad n\geq0.
				\]
				Combining
				Eqs.~\eqref{eq:high_frequency_decay} and
				\eqref{eq:low_frequency_decay} gives
				\[
				\frac{
					\|P_{\mathrm H}\boldsymbol e^{(n)}\|_{\mathcal Y}
				}{
					\|P_{\mathrm L}\boldsymbol e^{(n)}\|_{\mathcal Y}
				}
				\leq
				\left(
				\frac{q_{\mathrm H}}{q_{\mathrm L}}
				\right)^n
				\frac{
					\|P_{\mathrm H}\boldsymbol e^{(0)}\|_{\mathcal Y}
				}{
					\|P_{\mathrm L}\boldsymbol e^{(0)}\|_{\mathcal Y}
				}.
				\]
				Since $0\leq
				\frac{q_{\mathrm H}}{q_{\mathrm L}}
				<1$,
				we conclude that
				\[
				\lim_{n\to\infty}
				\frac{
					\|P_{\mathrm H}\boldsymbol e^{(n)}\|_{\mathcal Y}
				}{
					\|P_{\mathrm L}\boldsymbol e^{(n)}\|_{\mathcal Y}
				}
				=0.
				\]
				Therefore, under the common-eigenbasis and spectral-separation
				assumptions, the high-frequency component of the error is damped
				asymptotically faster than the low-frequency component.
			\end{proof}
			
			\begin{remark}[Connection to classical PDE iterative solvers]
				The assumptions of Theorem~\ref{thm:spectral_convergence_operator}
				are motivated by the spectral analysis of stationary iterations for elliptic problems. For example, for a constant-coefficient Poisson problem on a periodic regular grid, the Fourier modes diagonalize both the spatial operator $\mathcal L=-\Delta$ and the error-propagation operator
				\[
				\mathcal E
				=
				\mathrm{Id}_{\mathcal Y}-\mathcal B\mathcal A
				\]
				of Jacobi-type iterations. The factors $\eta_j$ then describe the damping of individual spatial frequencies. The condition
				\[
				q_{\mathrm H}<q_{\mathrm L}<1
				\]
				expresses the smoothing property that high-frequency errors are damped more strongly than low-frequency errors. This condition must be verified for the particular operator, iteration scheme, and frequency cutoff. For nonperiodic, variable-coefficient, or irregular-grid problems, the common spectral basis is an additional assumption and may not exist.
			\end{remark}
			
				

			\section{Comparison between SINO and Deep Unrolling}
			\label{Comparison}
			SINO is related to the general framework of algorithm unrolling
			\cite{monga2021algorithm}, in which a finite number of iterations of a numerical algorithm are represented as successive network layers and some of the algorithmic components are replaced by trainable modules. This construction incorporates prior knowledge of the underlying algorithm into the network architecture and can provide greater structural interpretability than a generic black-box model. In SINO, the Iterator follows this principle by unfolding a residual-correction scheme into a finite sequence of learned update steps.
			
			Deep unrolling methods nevertheless involve several practical challenges. Increasing the number of unrolled steps raises computational and memory costs and may make gradient-based optimization more difficult. Repeated compositions of similar update modules may also produce progressively
			smaller improvements as the network depth increases
			\cite{xu2025overview}. Moreover, unless appropriate constraints are imposed, the learned parameters need not satisfy the stability or convergence conditions of the original numerical algorithm. Like other data-driven models, unrolled networks may also experience performance degradation when the test distribution differs substantially from the training
			distribution. These limitations depend on the particular architecture, training procedure, and problem under consideration.
			
			SINO retains the residual-correction structure of algorithm unrolling while introducing a spectral Starter and performing the learned refinement in a latent representation. The Starter supplies a global spectral initialization, after which the Iterator applies multiscale residual corrections over a fixed number of steps. This separation allows the two
			modules to focus on initialization and refinement, respectively, while avoiding the need to obtain the final approximation from a single monolithic operator. The numerical results in
			Section~\ref{sec:model_eva} indicate that this design improves accuracy and convergence behavior on the problems considered in this work.
			
			The theoretical analysis further provides a compact-uniform approximation result for the Starter-Iterator decomposition under the fixed bounded-linear-operator assumptions stated in
			Theorem~\ref{thm:fsno_approx}. This result explains the approximation mechanism of the proposed decomposition but does not provide a general convergence guarantee for nonlinear practical Iterator blocks or for the complete state-dependent linearization process.
			
			{\section{Computational Complexity of SINO}
				\label{app:complexity}
				
				We analyze the computational complexity of the SINO to characterize its scalability. For a two-dimensional input discretized on an $n \times n$ grid, let $B$ denote the batch size, $C$ the latent channel width, $L$ the number of spectral layers in the starter, $K$ the number of retained Fourier modes per dimension, and $N$ the number of unrolled iterations. The total complexity is the sum of three components: lifting/projection, the spectral starter, and the iterative refinement.
				
				\textbf{Lifting and Projection.} The lifting and projection operators are implemented as pointwise linear transformations across the spatial grid. Their combined complexity scales as $\mathcal{O}(Bn^2C^2)$, where we assume the input and output channel dimensions are small constants relative to $C$. This cost is typically negligible compared to the starter and iterator.
				
				\textbf{Fourier-type starter.}
				The starter follows the architecture of the Fourier Neural Operator (FNO)~\cite{li2020fourier}. Each spectral layer consists of a fast Fourier transform (FFT), a spectral convolution on $K \times K$ modes, and an inverse FFT. The complexity for $L$ such layers is $\mathcal{O}\left( LBCn^2 \log n + LBC^2K^2 + LBn^2 C^2\right)$. This component captures global dependencies through its $\mathcal{O}(n^2 \log n)$ scaling.
				
				\textbf{CNN-based iterator.}
				The iterator operator $\mathcal{I}_\theta$ performs $N$ residual updates. Each iteration involves two local convolutional operators, $\mathcal{A}_\theta^{(n)}$ and $\mathcal{B}_\theta^{(n)}$, parameterized by $k \times k$ kernels (typically $k=3$). The floating-point operations (FLOPs) for one convolutional layer are approximately $2Bn^2C^2k^2$. Thus, for $N$ iterations, the refinement complexity is $ \mathcal{O}(NBn^2C^2k^2)$. Given that $k$ is a small fixed constant, the dominant term simplifies to $\mathcal{O}(NBn^2C^2)$.
				
				\textbf{Total Complexity.} Combining the above, the total inference complexity of SINO is given by:
				\[
				\text{FLOPs} = \mathcal{O} \left( LBCn^2 \log n + LBC^2K^2 + NBn^2C^2 \right).
				\]
				This formulation explains the empirical trade-off observed in Figure~\ref{fig:iter_effect}: increasing the number of iterations $N$ improves accuracy with a strictly linear increase in computational cost, while the number of trainable parameters remains nearly constant if weights are shared across iterations. In practice, a moderate $N$ provides significant error reduction over the starter's initialization while maintaining the efficiency inherent in local convolutional operations.
			}

			\section{Time Evolution Dataset}   \label{secATeq}
			\textbf{Incompressible Navier Stokes equation.} 
			To generate training data for learning the operator of the 2-D Navier–Stokes equations in vorticity form, we solve the equation on the unit torus with periodic boundary conditions.
			The initial vorticity $w_0(x)$ is sampled from a Gaussian random field
			\begin{equation*}
				w_0 \sim \mathcal{N}\!\left(0,\; 7^{3/2}(-\Delta+49I)^{-2.5}\right).
			\end{equation*}
			The forcing term is fixed as
			\begin{equation*}
				f(x) = 0.1 \bigl[\sin\!\bigl(2\pi(x_1+x_2)\bigr) + \cos\!\bigl(2\pi(x_1+x_2)\bigr)\bigr].
			\end{equation*}
			The equations are solved using the stream-function formulation with a pseudo-spectral method \cite{shen2011spectral}:
			a Poisson equation is solved in Fourier space to recover the velocity field from vorticity,
			the nonlinear term is evaluated in physical space, and dealiasing is applied.
			Time integration is performed using a Crank-Nicolson scheme for the viscous term and an explicit update for the nonlinear term,
			with a time step $\Delta t = 10^{-4}$.
			
			For the case with viscosity $\nu = 10^{-3}$, we generate $5{,}000$ data samples, while for $\nu = 10^{-4}$, we generate $10{,}000$ data samples. For each viscosity setting, the training dataset consists of $1{,}000$ trajectories, whereas the validation and test datasets each contain $200$ trajectories with independently 
			sampled initial conditions $w_0(x)$. Solutions are recorded every $\Delta t$ time units to construct the datasets used for training and evaluation.
			To ensure comparability, the resolution is fixed at $64 \times 64$ for both training and testing,
			and for smaller viscosities the final time horizon $T$ is reduced to avoid uncontrolled chaotic behavior.
			
			\textbf{Acoustic wave equation.} 
			To generate the dataset, we solve the above equation using a finite difference scheme with second-order central differencing in time and eighth-order finite differences in space.
			
			In scenarios where the wave speed distribution \(c(x,y)\) of the medium is fixed and the source location varies, we get the fixed velocity map as 
			\begin{equation*}
				c(x,y) = (1+\sin x\sin y), \quad (x, y) \in [0, \pi]^2.
			\end{equation*}
			The spatial grid consists of $N_x = N_y = 64$ equally spaced points. The initial condition $u_0(x)$ is chosen as a Gaussian source. Letting $t \in [0, 1]$ and $\Delta t = 1e-3$, we set 
			\[
			u(x, y,0) = \exp\!\left(-\frac{\|x-x_c\|_2^2 + \|y-y_c\|_2^2}{\sigma^2}\right),
			\]
			where the center $(x_c, y_c)$ is sampled randomly from a uniform distribution over the grid points. We consider the temporal domain $t \in [0,1]$ and adopt a time step size of 
			$\Delta t = 10^{-3}$, resulting in $1000$ discrete time steps. To construct the training and testing datasets, we uniformly subsample the solution every $20$ time steps, thereby obtaining temporal sequences length $50$. We use 1000 samples for model training and evaluation.
			
			In scenarios where the wave speed distribution \(c(x,y)\) varies while the source location remains fixed, we consider heterogeneous acoustic wave propagation in realistic breast tissue slices from the open dataset \cite{zeng2025openbreastus}, following the experimental setup in \cite{zeng2023neural}. Each breast slice provides a spatially varying speed map of shape $(480, 480)$ with a background sound speed of $1500 \mathrm{m/s}$. The computational domain is defined over a $24 \mathrm{cm} \times 24 \mathrm{cm}$ square region and discretized with a spatial step size $\Delta x = \Delta y = 0.05 \mathrm{cm}$.
			
			The wave equation is solved using a finite-difference time-domain scheme with second-order time stepping and eighth-order spatial differencing. We place four point sources uniformly along a circular ring of radius $9 
			\mathrm{cm}$ within the domain. The temporal domain is $t \in [0, 0.3] \mathrm{ms}$, with a time step size $\Delta t = 0.3\mu\mathrm{s}$, leading to $1000$ time steps in total. The source function is a Ricker wavelet with a central frequency of $0.5 \mathrm{MHz}$.
			
			The forward simulation returns pressure fields sampled every $20$ time steps, yielding temporal sequences of length $50$. For each breast slice, the recorded data has a final shape of $(50, 480, 480)$. We use four sources per simulation and adopt a record gap of $20$ for data collection. All simulations are conducted under a CFL number $C \le 1$ to ensure numerical stability.

			\textbf{Shallow-water equation.}  
			To generate $N_{\text{train}}$ samples for training, we consider a 2D radial dam-break problem on a square domain $\Omega=[-2.5,2.5]^2$. 
			The initial water height is set as 
			\[
			h(0,x,y)=
			\begin{cases}
				2.0, & r<\sqrt{x^2+y^2},\\[4pt]
				1.0, & r\geq \sqrt{x^2+y^2},
			\end{cases}
			\]
			where the radius $r$ is randomly sampled from a uniform distribution 
			$r \sim \mathcal{U}(0.3,0.7)$. The system is simulated using the finite volume solver provided by \texttt{PyClaw} \cite{ketcheson2012pyclaw}, with second-order accurate reconstruction in space and a TVD Runge-Kutta scheme for time integration. 
			The spatial grid is discretized into $256 \times 256$ cells, and the 
			solution is computed up to $T=1.0$ with a time step $\Delta t=10^{-3}$, resulting in $1000$ temporal snapshots. 
			For constructing the dataset, we uniformly subsample every $20$ snapshots to obtain temporal sequences of length $50$. 
			For each setting, we generate $N_{\text{train}}=1000$ training trajectories, and $N_{\text{vali}}=N_{\text{test}}=200$ trajectories for validation and testing, respectively, all with independently sampled initial radius.
			
			We summarize the considered time-evolution PDE benchmarks together with their corresponding operator mappings and representative input-output examples in Figure.~\ref{fig:table1}.
			\begin{figure}[htbp]
				\centering
				\includegraphics[width=1.0\linewidth]{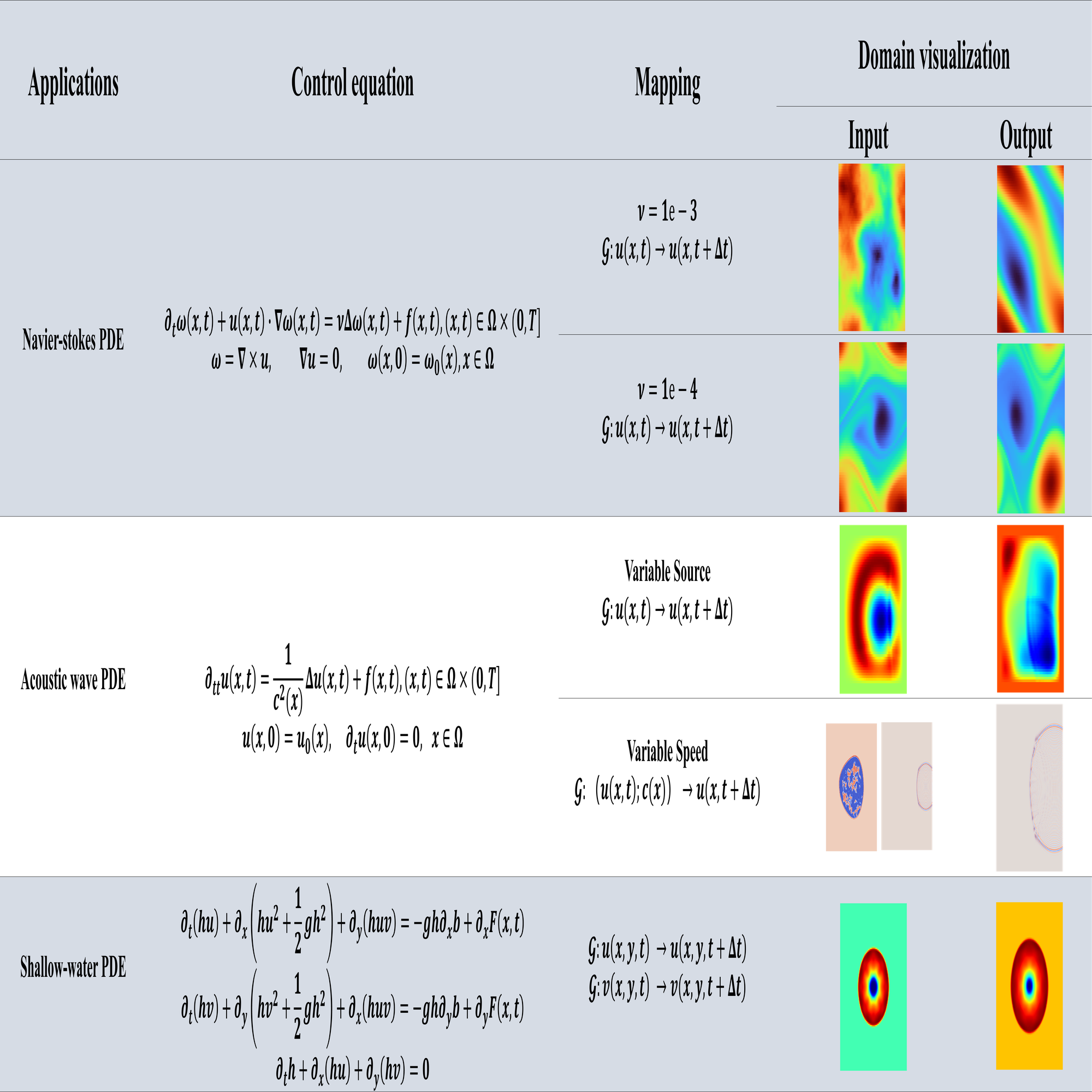}
				\caption{Summary of the time-evolution PDE benchmarks considered in this work, including the governing equations, operator mappings, and representative input-output field visualizations.
				}
				\label{fig:table1}
			\end{figure}
			
			\section{Additional results}
			We provide supplementary experimental results to further validate the robustness and generalization utility of the proposed SINO framework across various physical regimes. Specifically, wave propagation predictions on a high-wavenumber dataset are presented in Figure~\ref{fig:High_wave}, while the model's zero-shot super-resolution performance on the 1D Burgers’ equation is summarized in Table~\ref{tab:model_compare}. To further investigate long-term stability, we analyze error accumulation over 100 recursive autoregressive inference steps as illustrated in Figure~\ref{fig:long_horizon_advection_diffusion}, which demonstrates SINO's superior error resilience in extended temporal forecasting.
			
			\begin{figure}[htp!]
				\centering
				\includegraphics[width=0.85\linewidth]{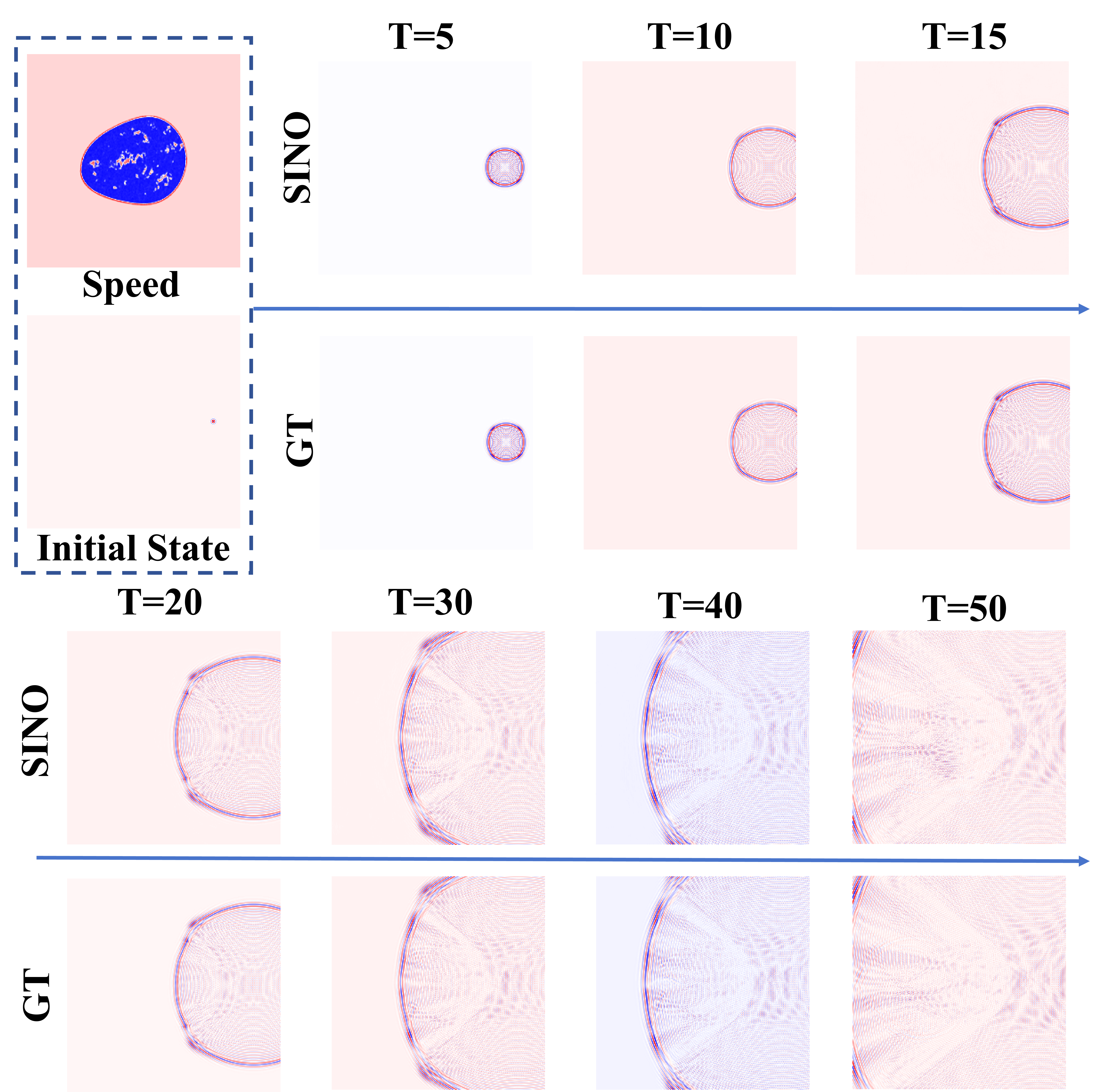}
				\vspace{0.3cm}
				\caption{\textbf{Wave propagation prediction on a high-wavenumber dataset.} 
					The leftmost panel (blue dashed box) illustrates the input configuration, including the velocity map and the source location. The panels to the right and below display the wavefield evolution over 50 time steps. Odd rows show the ground truth wavefields obtained from numerical simulations, while even rows present the corresponding predictions generated by the model. The comparison highlights the model’s ability to capture fine-scale oscillations and complex scattering patterns in high-frequency regimes.}
				\label{fig:High_wave}
			\end{figure}
			
			\begin{table*}[htbp]
				\centering
				\caption{Benchmarks on 1D Burgers’ equation for zero-shot super resolution. We train the SINO model with different numbers of iterations $N$ on 1D Burgers data with size$=256$ and evaluated at various higher resolution, without seeing any higher resolution data. The best results are highlighted in boldface.}
				\label{tab:model_compare}
				\renewcommand{\arraystretch}{1.2}
				\setlength{\tabcolsep}{10pt}
				\begin{tabular}{lcccccc}
					\hline
					\textbf{Model} & \textbf{R=256} & \textbf{R=512} & \textbf{R=1024} & \textbf{R=2048} & \textbf{R=4096} & \textbf{R=8192} \\
					\hline
					
					FNO 
					& 3.6038e-03 & 4.1226e-03 & 4.6048e-03 & 4.9082e-03 & 5.0679e-03 & 5.1495e-03 \\
					
					SINO $N=1$ 
					& 2.6836e-03 & 3.0529e-03 & 3.4430e-03 & 3.6811e-03 & 3.8064e-03 & 3.8704e-03 \\
					
					SINO $N=2$ 
					& 1.8061e-03 & 1.9723e-03 & 2.1593e-03 & 2.2689e-03 & 2.3270e-03 & 2.3568e-03 \\
					
					SINO $N=4$ 
					& 5.6948e-04 & 6.3532e-04 & 7.0056e-04 & 7.3600e-04 & 7.5465e-04 & 7.6419e-04 \\
					
					SINO $N=8$ 
					& 4.8622e-04 & 5.1652e-04 & 5.5206e-04 & 5.7300e-04 & 5.8421e-04 & 5.8999e-04 \\
					
					SINO $N=16$ 
					& 4.7280e-04 & 4.9733e-04 & 5.2427e-04 & 5.4068e-04 & 5.4952e-04 & 5.5409e-04 \\
					
					\textbf{SINO $N=32 $}
					& \textbf{3.5319e-04} & \textbf{3.6847e-04} & \textbf{3.8411e-04} & \textbf{3.9352e-04} & \textbf{3.9861e-04} & \textbf{4.0124e-04} \\
					
					\hline
				\end{tabular}
			\end{table*}
			
			\begin{figure}
				\centering
				\includegraphics[width=0.5\linewidth]{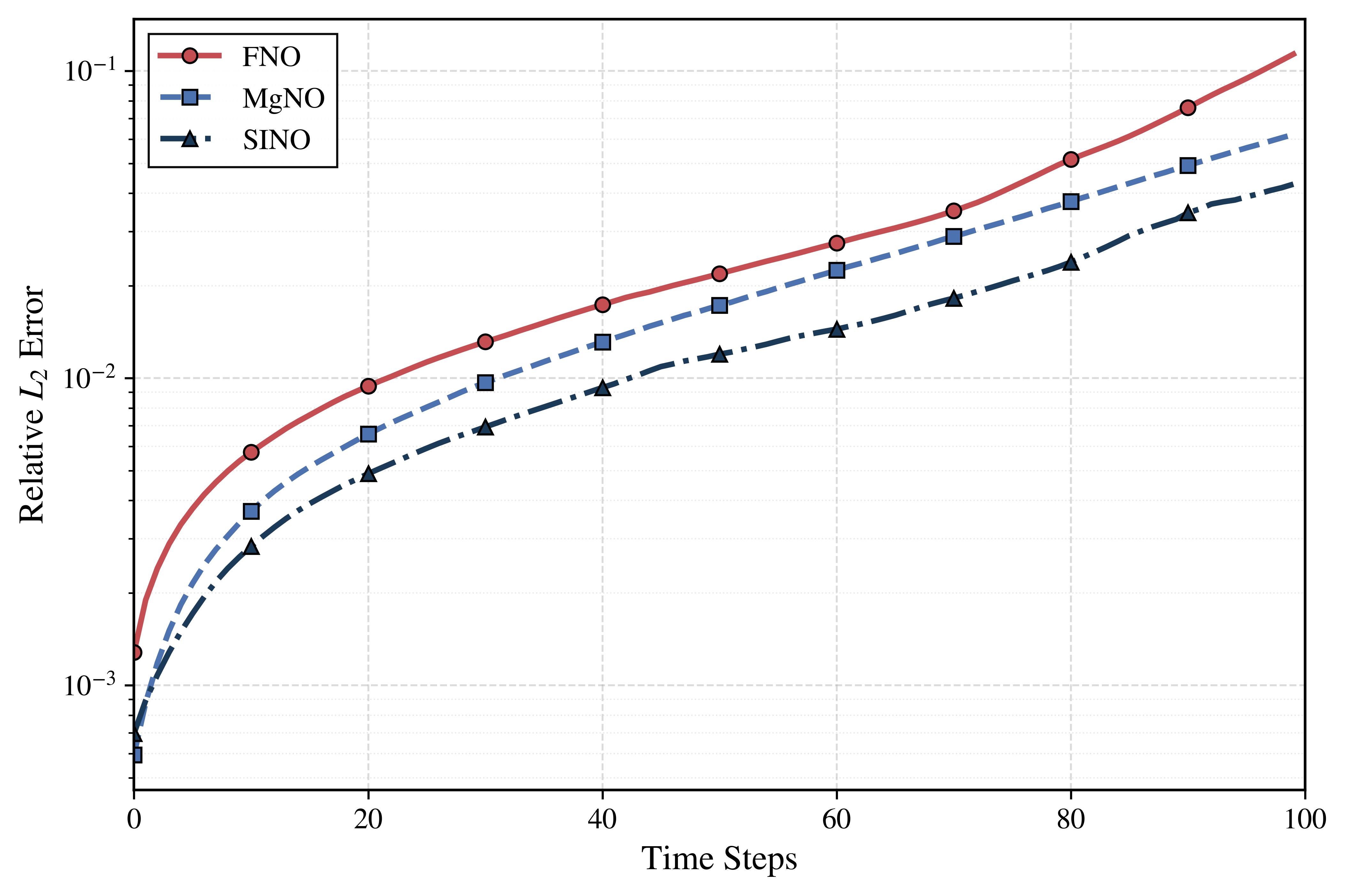}
				\caption{
					Long-horizon error accumulation on the linear advection benchmark.
					The models are evaluated for $100$ autoregressive inference steps, and the
					relative $L_2$ error is reported at different time steps. Although the
					prediction error increases over time for all methods due to accumulated
					autoregressive error, SINO consistently achieves the lowest error throughout
					the entire horizon. Compared with FNO and MgNO, SINO exhibits slower error
					growth and better long-time stability.
				}
				\label{fig:long_horizon_advection_diffusion}
			\end{figure}

		\end{document}